\documentclass[11pt]{amsart}
\usepackage{amscd,amssymb,amsopn,amsmath,amsthm,mathrsfs,graphics,amsfonts,enumerate,verbatim,calc
}
\usepackage{bbm}
\usepackage[all,cmtip]{xy}
\usepackage[all]{xy}
\usepackage{tikz}
\usepackage{lscape}
\usepackage{enumitem}
\usetikzlibrary{matrix,arrows,decorations.pathmorphing}
\usepackage{hyperref}
\hypersetup{colorlinks=true,linkcolor={blue}}
\usepackage{cleveref}
\usepackage{scalerel}
\usepackage{stackengine,wasysym}
\usetikzlibrary {positioning}
\usetikzlibrary{patterns}
\usetikzlibrary{calc}
\definecolor {processblue}{cmyk}{0.96,0,0,0}

\usepackage{subfigure}

\usepackage{comment} 

\usepackage{ dsfont }
\usepackage{color}

\usepackage[OT2,OT1]{fontenc}
\newcommand\cyr{%
\renewcommand\rmdefault{wncyr}%
\renewcommand\sfdefault{wncyss}%
\renewcommand\encodingdefault{OT2}%
\normalfont
\selectfont}
\DeclareTextFontCommand{\textcyr}{\cyr}

\usepackage{amssymb,amsmath}
\DeclareFontFamily{OT1}{rsfs}{}
\DeclareFontShape{OT1}{rsfs}{n}{it}{<-> rsfs10}{}
\DeclareMathAlphabet{\mathscr}{OT1}{rsfs}{n}{it}

\numberwithin{equation}{section}
\newtheorem{theorem}{Theorem}[section]
\newtheorem{lem}[theorem]{Lemma}

\newtheorem{Corollary}[theorem]{Corollary}

\newtheorem{conjecture}[theorem]{Conjecture}

\newtheorem{prop}[theorem]{Proposition}

\newtheorem{Notation}[theorem]{Notation}

\theoremstyle{definition}
\newtheorem{defn}[theorem]{Definition}
\theoremstyle{remark}
\newtheorem{remark}[theorem]{Remark}
\newtheorem{Remark}[theorem]{Remark}

\newtheorem{example}[theorem]{Example}

\newcommand{\F}{\mathcal{F}}

\newcommand{\ssi}[1]{{\rm ssi}(#1)}
\renewcommand{\tilde}{\widetilde}

\newcommand{\ovl}[1]{\overline{#1}}
\newcommand{\ul}[1]{\underline{#1}}
\newcommand{\wdt}[1]{\widetilde{#1}}

\newcommand{\tovl}[1]{$\overline{\mbox{#1}}$}

\newcommand{\lk}{\operatorname{lk}}
\newcommand{\Ass}{\operatorname{Ass}}

\newcommand{\pol}{\operatorname{pol}}
\newcommand{\grade}{\operatorname{grade}}

\newcommand{\reg}{\operatorname{reg}}

\newcommand{\pd}{\operatorname{pd}}

\newcommand{\Skel}{\operatorname{Skel}}

\renewcommand{\pol}{\operatorname{pol}}

\newcommand{\p}{\mathfrak{p}}
\newcommand{\q}{\mathfrak{q}}
\newcommand{\m}{\mathfrak{m}}

\def\ZZ{{\mathbb Z}}
\def\NN{{\mathbb N}}

\def\grade{{\rm grade}}
\def\supp{{\rm supp}}
\def\Llra{\Longleftrightarrow}
\def\llra{\longleftrightarrow}
\def\Lra{\Longrightarrow}
\def\lra{\longrightarrow}
\def\Lla{\Longleftarrow}

\newcommand{\be}{\begin{equation*}}
	\newcommand{\ee}{\end{equation*}}
\newcommand{\bee}{\begin{equation}}
	\newcommand{\eee}{\end{equation}}

\renewcommand{\bar}{\overline}

\begin{document}
\title[Ordinary and Symbolic Powers of Matroids via Polarization]{Ordinary and Symbolic Powers of Matroids via Polarization}

\author[Lyle]{Justin Lyle}
\email[Justin Lyle]{jll0107@auburn.edu}
\urladdr{https://jlyle42.github.io/justinlyle/}
\address{Department of Mathematics and Statistics \\ 221 Parker Hall\\
	Auburn University\\
	Auburn, AL 36849}

\author[Mantero]{Paolo Mantero}
\email[Paolo Mantero]{pmantero@uark.edu}
\address{Department of Mathematical Sciences\\
University of Arkansas\\
Fayetteville, AR 72701 USA}

\begin{abstract}
In this paper, we propose a uniform approach to tackle problems about squarefree monomial ideals whose powers have good properties. 
We employ this approach to achieve a twofold goal: (i) recover and extend several well--known results in the literature, especially regarding Stanley--Reisner ideals of matroids, and (ii) provide short, elementary proofs for these results. 
Among them, we provide simple proofs of two celebrated results of Minh and Trung, Varbaro, and Terai and Trung \cite{TT12} elegantly characterizing the Cohen-Macaulay property, or even Serre's condition $(S_2)$, of symbolic and ordinary powers of squarefree monomial ideals in terms of their combinatorial (matroidal) structure. Our work relies on the interplay of several combinatorial and algebraic concepts, including dualities, polarizations, Serre's conditions,  matroids,  Hochster-Huneke graphs, vertex decomposability, and careful choices of monomial orders.  
\end{abstract}

\maketitle

\section{Introduction}

Symbolic and ordinary powers of ideals have received a great deal of attention both historically and in recent years, owed significantly to the wealth of geometric and algebraic information that they contain. 
One of the very first published results in this direction was proved, for general radical ideals, by Cowsik and Nori \cite{CN76} (see also Dade's Ph.D. thesis \cite{Da60}). 
\begin{theorem}[Cowsik--Nori 1976, Dade 1960]\footnote[1]{This is how the theorem is stated in \cite{CN76}. However, it is not hard to see that their proof essentially proves the following more general statement: Let $R$ be a Cohen--Macaulay positively graded $k$-algebra, and let $I$ be a graded ideal that is generically a complete intersection. Then $I^m$ is Cohen--Macaulay for every $m\geq 1$ if and only if $I$ is a complete intersection. See also \cite{Wa78}, \cite{AV79} or \cite{HO80}.}
	Let $R=k[x_1,\ldots,x_n]$ and $I$ a radical homogeneous ideal in $R$. Then $I^m$ is Cohen--Macaulay for every $m\geq 1$ if and only if $I$ is a complete intersection.
\end{theorem}

In general, geometric information of ordinary powers and algebraic information of symbolic powers of ideals can be quite hard to compute, and therefore such powers of squarefree monomial ideals have been widely investigated for the past 50 years as their combinatorics give a natural foothold on these problems. Here, we propose a polarization--based approach to these investigations. Polarizations of monomial ideals have been employed extensively in the literature, however their applications to the study of symbolic powers of squarefree monomial ideals has been limited by some drawbacks. We illustrate a notable one: Let $k$ be a field, let $\Delta$ be a simplicial complex on $[n]$, let  $I=I_{\Delta}\subseteq R=k[x_1,\ldots,x_n]$ be the Stanley--Reisner ideal of $\Delta$, and let  $I^{(m)}$ be the {\em $m$-th symbolic power} of $I$. (see Definition \ref{Def-Symb}.) The (standard) polarization of  $I^{(m)}$ is the Stanley--Reisner ideal of a simplicial complex $\Delta^{(m)}$ on $[n]\times [m]$; if $\dim(\Delta)=d-1$, then $\dim(\Delta^{(m)})=d + n(m-1)-1$. The large size of the vertex set and the large dimension of $\Delta^{(m)}$ pose serious challenges to effectively decoding the combinatorics of $\Delta^{(m)}$. 

Our approach consists of a rigid, general part, and a flexible, more specific component. The general part is centered upon the study of the so--called {\em naive dual} $(\Delta^{(m)})^*$ of $\Delta^{(m)}$ and its Alexander dual. An immediate advantage of this approach is that $\dim((\Delta^{(m)})^*)=\dim(\Delta^*)$, and the set $\mathcal F((\Delta^{(m)})^*$ of the facets  of $(\Delta^{(m)})^*$ is very easily described starting from  $\mathcal F(\Delta)$. The second component is usually determined by the specific problem at hand and may vary quite extensively. (see the last paragraph of this introduction.)

While Alexander duals of  polarizations of powers of  very {\em special} classes of squarefree monomial ideals have been investigated before, e.g. in \cite{Fr08,AFH24}, our slightly twisted approach via the naive dual turns out to be quite powerful to study {\em general} questions about squarefree monomial ideals whose powers have good properties. As an illustration, we recover, and often slightly strengthen, a number of theorems in the literature, and we prove a few new results. \medskip

For instance,  in 2011, Minh and Trung \cite{MT11} and Varbaro \cite{Va11} proved -- independently, and with very different techniques -- the following elegant combinatorial characterization of when Stanley--Reisner ideals have Cohen--Macaulay symbolic powers: {\em Let $R$ be a polynomial ring over a field $k$ and let $I_{\Delta}$ be the Stanley--Reisner ideal associated to a simplicial complex $\Delta$. Then $I_{\Delta}^{(m)}$ is Cohen--Macaulay for every $m\geq 1$ if and only if $\Delta$ is a matroidal simplicial complex.}

In 2012, Terai and Trung \cite{TT12} strengthened the above theorem, and, additionally, proved a version of it for ordinary powers. With our approach, we provide new proofs recovering these results and strengthen them by adding new characterizations for these conditions:
\begin{theorem}\label{cmsymbintro}

The following are equivalent:

\begin{enumerate}

\item[$(1)$] $R/I_{\Delta}^{(m)}$ is Cohen-Macaulay for every $m \ge 1$.

\item[$(2)$] The simplicial complex $\Delta^{(m)}$ is vertex decomposable for every $m \ge 1$.	

\item[$(3)$] $R/I_{\Delta}^{(m)}$  is Cohen-Macaulay  for some $m \ge 3$.

\item[$(4)$] $R/I_{\Delta}^{(m)}$ satisfies Serre's condition $(S_2)$  for some $m \ge 3$.

\item[$(5)$] $\Delta$ is a matroid.

\end{enumerate}
\end{theorem}

\begin{theorem}\label{cmordintro}

 The following are equivalent:

\begin{enumerate}

\item[$(1)$] $R/I_{\Delta}^{m}$ is Cohen-Macaulay for every $m \ge 1$.

\item[$(2)$] $R/I_{\Delta}^{m}$  is Cohen-Macaulay  for some $m \ge 3$.

\item[$(3)$] $R/I_{\Delta}^{m}$ satisfies  Serre's condition $(S_2)$  for some $m \ge 3$.

\item[$(4)$] $\Delta$ is a matroid satisfying the K\"onig property (see Remark \ref{Konig})

\item[$(5)$] $\Delta$ is a complete intersection.

\end{enumerate}

\end{theorem}

The above-mentioned theorem of Minh and Trung and Varbaro is the equivalence of (1) and (5) in Theorem \ref{cmsymbintro}. Terai and Trung proved the equivalence of all conditions except (2) \cite[Thm~3.6]{TT12}.  Similarly, the theorem of Cowsik and Nori proves the equivalence of (1) and (5) in Theorem \ref{cmordintro} under the more general assumption that $I=\sqrt{I}$ is any radical ideal. Terai and Trung proved the equivalence of all conditions in Theorem \ref{cmordintro} except (4) \cite[Thms~4.3 and 4.6]{TT12}. In fact, parts (2) in Theorem \ref{cmsymbintro} and (4) in Theorem \ref{cmordintro} are new conditions, obtained specifically with our approach. 

In a similar vein, we recover the main result of \cite{MT17}, i.e. an explicit formula,  first found by Minh and Trung, for the Castelnuovo--Mumford regularity of the symbolic powers of Stanley--Reisner ideals associated to matroidal simplicial complexes, see \cite[Thm~4.5]{MT17}:
\begin{theorem}\label{regintro}
Let $\Delta$ be a matroid, then 
$$\reg(I_{\Delta}^{(m)}) = (m-1)c(\Delta) + r({\rm core}(\Delta)) + 1,$$
where $c(\Delta)$ is the circumference of $\Delta$, and ${\rm core}(\Delta)$ is the core of $\Delta$.
\end{theorem}

Additionally, our methods allow us to recover the main result of \cite{MTT19}, i.e. the following theorem of Minh, Terai, and Thuy characterizing the level property of quotients by symbolic or ordinary powers of Stanley--Reisner rings, see \cite[Theorem 4.3]{MTT19}. Recall that a graded quotient of $R$ is a {\em level algebra} if it is Cohen-Macaulay and its canonical module is generated in single degree. 

\begin{theorem}\label{levelintro}
The following are equivalent:
\begin{enumerate}

\item[$(1)$] $I_{\Delta}$ is a complete intersection generated in a single degree.
\item[$(2)$] $R/I_{\Delta}^m$ is a level algebra for all $m \ge 1$.
\item[$(3)$] $R/I_{\Delta}^m$ is a level algebra for some $m \ge 3.$
\item[$(4)$] $R/I_{\Delta}^{(m)}$ is a level algebra for all $m \ge 1$.
\item[$(5)$] $R/I_{\Delta}^{(m)}$ is a level algebra for some $m \ge 3$.

\end{enumerate}
\end{theorem}

One interesting consequence of Theorems \ref{cmsymbintro}, \ref{cmordintro}, \ref{levelintro} is that, if the 3rd (ordinary or symbolic) power of $I_{\Delta}$ has one of the ``good" properties stated in the theorems, then all such powers have the same property, and the combinatorial structure of $\Delta$ is strongly constrained.  The proofs of all the versions of these theorems required the introduction of a good amount of new and interesting machinery, in some cases spanning over multiple works. For instance, for Theorem \ref{cmsymbintro}, \cite{MT11} and \cite{TT12} heavily rely on Takayama's Formula, which is a multi-graded version of Hochster's formula providing a combinatorial characterization of the graded pieces of the local cohomology modules $H_{(x_1,\ldots,x_n)}^i(R/I_{\Delta}^{(m)})$, and previous results proved in \cite{MTr09}. \cite{MT11} also employs techniques from linear programming, while \cite{Va11} uses deep results in the theory of blowup algebras.  

\bigskip

Our uniform approach allows us to recover Theorems \ref{cmsymbintro}, \ref{cmordintro},  \ref{regintro} and \ref{levelintro} via {\em short and elementary} proofs. Additionally, the elementary combinatorics of our proofs provide concrete reasons as to why the characterizations require third or higher symbolic powers while the second is insufficient; see e.g. Remark \ref{m=3}. In regard to the flexible component of the approach, for the proofs of \ref{cmsymbintro}, and \ref{cmordintro}  we employ the so--called Hochster--Huneke graph and a careful choice of total orders of the vertex set $[n]\times [m]$ and  $\mathcal F((\Delta^{(m)})^*)$. As a consequence, we can finish the proof via elementary techniques. The proof of Theorem \ref{regintro} relies on Theorem \ref{cmsymbintro} and very delicate choices of orderings. (see Lemma \ref{C_M}.) 

On the other hand, to illustrate the potential  flexibility in the second part of the approach, for the proof of Theorem \ref{levelintro} we have employed elementary simplicial homology techniques, namely we calculate explicit homology elements that represent non--zero graded Betti numbers.

In addition to the above, we recover and extend a theorem by Herzog, Takayama and Terai (Theorem \ref{mixedSymbPwrs}), we provide a new, precise characterization of when is $R/I_\Delta^{(2)}$ Cohen--Macaulay (Proposition \ref{2ndSymbPwr}), and we prove that when $\Delta$ is a matroidal simplicial complex, then the polarization of $I_\Delta^{(m)}$ is glicci for every $m\geq 1$. (Corollary \ref{glicci}.) All our results are independent of and complementary to the ones in \cite{MN25}.
 
Section 2 contains preliminaries and basic notation. In Section 3 we prove Theorem \ref{cmsymbintro} and provide new results obtained thanks to our approach. In Section 4 we prove Theorem \ref{cmordintro}. In Section 5 we prove Theorem \ref{regintro} and in Section 6 we prove Theorem \ref{levelintro}.

\section{Preliminaries and notation}\label{not}

For the purpose of this paper, we will use the following notation.
\begin{Notation}
 $k$ is a fixed field of {\em any} characteristic, and $R:=k[x_1,\dots,x_n]$.
 
 For any monomial ideal $J$ in $R$, the set $G(J)$ consists of the unique minimal generating set consisting of monomials.

$\Delta$ is a simplicial complex on the vertex set $[n]$, and $\mathcal{F}(\Delta)$ is the set of its facets. 

$I_{\Delta}:=\left(\prod_{i \in \sigma} x_i \mid \sigma \in [n] \setminus \Delta \right)$ is the Stanley-Reisner ideal of $\Delta$ in $R$, and we write $k[\Delta]:=R/I_\Delta$. 

We set $d:=\dim(k[\Delta])=\dim(\Delta)+1$, $c:=n-d$. 

Given $F \in \mathcal{F}(\Delta)$, we let $\p_F:=(x_i \mid i \in [n]-F)$.

When we say that \ul{$\Delta$ has an algebraic property $\mathcal P$} we mean that $k[\Delta]$ has property $\mathcal P$. 
 
\end{Notation}

\begin{Notation}\label{order}
In a few occasions we will need a total order $s_1,s_2,\ldots$ on the elements of some finite set $S$. This requires one to philosophically choose if one should regard $s_1$ as the ``largest" or the ``smallest" element in $S$ under the chosen order.

For consistency with the above notation, where smaller indices correspond to earlier elements, we will write $s_1<s_2<\ldots$, i.e. we write elements in increasing order. As a consequence, all inequalities will follow this convention, which can be summarized as ``smaller indices, variables, and monomials come first".

For instance, under this convention, to ensure that $s_1=x^2y$ comes before $s_2=xy^2$ under the standard lex order where $x$ has bigger weight than $y$, we use the order $x<y$ and say that under this Lex order we have $x^2y<xy^2$. 
\end{Notation}

Since our results are not affected by the existence of vertices in $[n]$ which are not contained in any face, we will assume that $\{i\}\in \Delta$ for every $i\in [n]$.

It is well known that $I_{\Delta}=\bigcap_{F \in \mathcal{F}} \p_{F}$, so for $F\in \F(\Delta)$, one has $\p_F\in \Ass_R{k[\Delta]}$, i.e. $\p_F$ is an associated prime of $R/I_{\Delta}$
The \textit{Alexander dual} $\Delta^{\vee}$ of the simplicial complex $\Delta$ is the simplicial complex $\Delta^{\vee}=\langle \sigma \subseteq [n] \mid [n]-\sigma \notin \Delta \rangle$.
The \textit{Alexander dual} of $I_\Delta$ is the ideal ${I_{\Delta}}\!^{\vee}:=I_{\Delta^{\vee}}=\left( \prod_{x_i \in \p_F} x_i \mid F \in \mathcal{F}(\Delta) \right)$.   

For any subset $F\subseteq [n]$, we write $F^*$ for the complement of $F$ in $[n]$, i.e. $F^*:=[n]-F$. We let $\Delta^*$ be the simplicial complex whose facets are the complement of the facets of $\Delta$, i.e. $\F(\Delta^*)=\{F^*\,\mid\, F \in \mathcal{F}(\Delta)\}$. We refer to $\Delta^*$ as the (naive) {\em dual} of $\Delta$. In general, $\Delta^*$ is {\em not} the Alexander dual of $\Delta$. 
In view of the bijection 
$$
\F(\Delta^*) \llra \Ass_R(k[\Delta]) \qquad \text{ given by }\qquad (i_1,\dots,i_c) \longmapsto (x_{i_1},\dots,x_{i_c})
$$
we will often silently identify $\F(\Delta^*)$ with $\Ass_R(k[\Delta])$.

\subsection{Complete intersections}
In view of Theorem \ref{cmordintro}, we recall that a proper homogeneous ideal $I$ in a graded ring $S$ is a {\em complete intersection of grade $r$} if it can be generated by a homogeneous regular sequence of $r$ elements, i.e. if one can write $I=(f_1,\ldots,f_r)$ where the $f_i$ are homogeneous, $f_1\neq 0$ and $f_i$ is a non-zero divisor of $S/(f_1,\ldots,f_{i-1})$, for every $i\geq 1$.

The {\em grade} of $I$, denoted ${\rm grade}(I)$, is the largest grade of a complete intersection in $I$.  When $I$ is a squarefree monomial ideal, one can easily characterize complete intersections combinatorially.

\begin{lem}\label{ci}
The following are equivalent for a simplicial complex $\Delta$:
\begin{enumerate}
	\item[$(1)$] $I_{\Delta}$ is a complete intersection of grade $c$;
	\item[$(2)$] There is a partition $V_1,\ldots,V_c$ of $[n]$ so that if one lets $\Gamma_i$ be the complex whose facets are the (isolated) vertices of $V_i$ ,then $\Delta^*$ is the join $\Gamma_1\star \cdots \star\Gamma_c$.
\end{enumerate} 
\end{lem}
Following \cite{TT12}, we say that $\Delta$ a {\em complete intersection} if any of the conditions of Lemma \ref{ci} is satisfied. E.g. if $\F(\Delta)=\{1234, 1235, 1246, 1256, 1346, 1356, 2347, 2357, 2467, 2567, 3467, 3567\}$, then $\Delta$ is a complete intersection. In fact, the partition $V_1:=\{1,7\}$, $V_2:=\{2,3,6\}$ and $V_3:=\{4,5\}$ satisfies Lemma \ref{ci}(2), so $I_{\Delta}=(x_1x_7, x_2x_3x_6, x_4x_5)$ is a complete intersection of grade 3.
 
 \subsection{Cohen--Macaulayness and Serre's conditions}
If $I$ is any homogeneous ideal in $R$, the {\em depth} of $R/I$ is the common length of any maximal regular sequence in the homogeneous maximal ideal of $R/I$. One always has ${\rm depth}(R/I)\leq \dim(R/I)$; the ring $R/I$ is called {\em Cohen--Macaulay} if equality is achieved. For instance, it is well-known that if $\Delta$ is shellable then $R/I_{\Delta}$ is Cohen--Macaulay (over any field $k$). 

We say $R/I$ is {\em level} if it is Cohen-Macaulay and one of the following two equivalent conditions hold: the canonical module of $R/I$ is generated in a single degree,  or all minimal generators of the last module in the minimal graded free resolution of $R/I$ over $R$ have the same degree. E.g. if $\F(\Delta)=\{1, 2, 3\}$, then $I_{\Delta}=(x_1x_2, x_1x_3, x_2x_3)$, and $\Delta$ is level. In contrast, if $\F(\Delta)=\{134, 135, 145, 245\}$, then $I_{\Delta}=(x_1x_2, x_2x_3, x_3x_4x_5)$, and $\Delta$ is Cohen--Macaulay but not level. 

One weakening of the Cohen--Macaulay property is the so--called {\em Serre's condition} $(S_{\ell})$. For any $\ell \in \NN_0$, the ring $R/I$ satisfies $(S_{\ell})$ if,  for all $\p\in V(I)$, the localization $(R/I)_\p$ satisfies ${\rm depth}((R/I)_\p)\geq \min \{\dim(R/I)_\p, \ell\}$. It follows immediately from the definition that if $R/I$ is $(S_\ell)$ then $R/I$ is Cohen--Macaulay in codimension $\ell$, i.e. $R_\p/I_\p$ is Cohen--Macaulay whenever  $\dim(R_\p/I_\p)\leq \ell$. So Serre's condition $(S_\ell)$ can be viewed as a strengthening of ``being Cohen--Macaulay in codimension $\ell$". It follows from the definition that $R/I$ is Cohen--Macaulay if and only if $R/I$ has Serre's condition ($S_\ell$) for every $\ell\geq 1$. Also, $k[\Delta]$ is ($S_1$) for any $\Delta$, and if $k[\Delta]$ has Serre's condition ($S_\ell$) for some $\ell\geq 2$, then $\Delta$ is pure. Background on these conditions outside of the squarefree monomial ideal setting can be found, for instance, in \cite{BH93}.

For a Stanley--Reisner ideal $I_{\Delta}$, there are useful combinatorial characterizations of each of the above properties. For instance, part (a) of the following theorem is a celebrated result by Reisner combinatorially characterizing the Cohen--Macaulay property, {\cite[Thm~1.]{Re76}. Parts (b) and (c) can be found, for instance, in {\cite{Te07}} or {\cite[Prop~2.4,~2.5]{HL21}}.

\begin{theorem} Let $\Delta$ be a simplicial complex.
\begin{enumerate}
\item[$(a)$] $\Delta$ is Cohen--Macaulay if and only if $ \widetilde{H}_i({\rm \lk}_{\Delta}(F); k)=0$ whenever $F \in \Delta$ and $i < \dim({\rm lk}_{\Delta}(F))$.
\item[$(b)$] ${\rm depth}(k[\Delta])\geq t$  if and only if $\widetilde{H}_{i-1}({\rm \lk}_{\Delta}(F); k)=0$ for all $F\in \Delta$ with $i+|F|<t$.
\item[$(c)$] $\Delta$ has Serre's condition ($S_\ell$) for some $\ell\geq 2$ if and only if 
$\widetilde{H}_{i-1}({\rm \lk}_{\Delta}(F); k)=0$ whenever $i + |F| \leq \dim(\Delta)$ and 
$0\leq i < \ell$.
\end{enumerate}
\end{theorem}

Similarly, a quick application of a variant of Hochster's formula (\cite[Theorem 5.5.1]{BH93}) gives the following characterization of the level property:

\begin{theorem}\label{char-level}
$\Delta$ is level if and only if it is Cohen-Macaulay and $|V|=|W|$ for any $V,W \subseteq [n]$ with $\tilde{H}_{|V|-c-1}(\Delta|_V) \ne 0$ and $\tilde{H}_{|W|-c-1}(\Delta|_W) \ne 0$.     
\end{theorem}

We now briefly recall the notions of ordinary and symbolic powers of an ideal. Symbolic powers can be defined in greater generality, but we record here only the definition needed in our setting.
\begin{defn}\label{Def-Symb}
 (1) If $I=(u_1,\ldots,u_r)$ is any ideal, the {\em $m$-th (ordinary) power} of $I$ is the ideal $I^m$ generated by the products of $m$ of the $u_i$'s, i.e. $I^m=(u_{i_1}\cdots u_{i_m}\,\mid\, i_j\in [r]\text{ for all } j)$.  \\
 (2) For any $m\in \ZZ_+$, the {\em $m$-th symbolic power} of a Stanley--Reisner ideal $I_{\Delta}$ is the ideal $$I_{\Delta}^{(m)}=\bigcap_{F\in \mathcal{F}(\Delta)}\left({\p_F}^m\right).$$
\end{defn}
It is immediately seen that $I_{\Delta}^{m}\subseteq I_{\Delta}^{(m)}$, and, in general, the inclusion is strict. E.g. if $\F(\Delta)=\{1, 2, 3\}$, then $I_{\Delta}=(x_1x_2, x_1x_3, x_2x_3)$, and $I_\Delta^2=(x_1^2x_2^2, x_1^2x_2x_3, x_1^2x_3^2, x_1x_2^2x_3, x_1x_2x_3^2, x_2^2x_3^2)\subsetneq I_\Delta^{(2)} = (x_1x_2x_3, x_1^2x_2^2, x_1^2x_3^2, x_2^2x_3^2)$.

\subsection{Matroids}

\begin{Remark}
In some of the proofs we will show that a given simplcial complex is matroidal. In this situation, following a number of authors in the literature, e.g. \cite{MT11}, \cite{Va11}, \cite{TT12}, \cite{MT17}, \cite{MTT19}, it seems more efficient to identify a matroid with its independence complex so we can consider matroids as a sub-class of all simplicial complexes. 

On the other hand, statements about matroids will be written, naturally, using the matroidal language, e.g. we will use the word ``circuit" instead of ``minimal non--face"; or ``coloop" for ``cone point". 
\end{Remark}

One can define matroids via the following exchange conditions.
\begin{defn}\label{matroid} 
	A \textit{matroid} is a simplicial complex $\Delta$ satisfying any of the following conditions:
\begin{enumerate}
	\item[$(1)$] For any $F,G \in \F(\Delta)$, and any $x \in G-F$, there is $y \in F-G$, so that $(G - \{x\}) \cup \{y\} \in \F(\Delta)$.
 
    \item[$(2)$] For any $F,G \in \F(\Delta)$, and any $x \in G-F$, there is $y \in F-G$, so that $(G - \{x\}) \cup \{y\}$ and $(F - \{y\}) \cup \{x\}$ are both in $\F(\Delta)$.
	\end{enumerate}
When these conditions are satisfied, the facets of $\Delta$ are also called the {\bf bases} of the matroid $\Delta$. The {\em rank} $r(\Delta)$ of $\Delta$ is the size of any of its bases. 
\end{defn}
For basic properties and language of matroids we refer the reader to \cite{Ox92}. 
A first example of a matroid is the {\em uniform matroid} of rank $c$ on $[n]$, whose bases are all the subsets of $[n]$ of size $c$.

It is easily seen that a simplicial complex $\Delta$ is a matroid if and only if $\Delta^*$ is a matroid, so we will frequently work in practice with $\Delta^*$, which is called the {\em dual matroid} of $\Delta$. 

Because of the identification of $\mathcal{F}(\Delta^*)$ and $\Ass_R(k[\Delta])$, we will frequently think of the above exchange conditions as occurring directly on the associated primes of $\Delta$. To this end,  for a prime ideal $\p_F=(x_1,\dots,x_c)$ generated by variables, we will often write $\p_F-(x_j)$ to denote the ideal generated by the $x_i$ with $i\neq j$, i.e. 
$$\p_F-(x_j):=\left( x_i \,\mid\, i=1,\ldots,c,\;i \ne j\right).$$

\subsection{Vertex decomposable simplicial complexes}

We now recall two definitions. The first one, vertex decomposability, is a well--studied combinatorial notion; the second one is a weakening of the notion of a matroid, due to Kokubo and Hibi.
\begin{defn}\label{weakpoly}
	A simplicial complex $\Delta$ is 
	\begin{itemize}
		\item[$($a$)$]  {\em vertex decomposable} if either $\Delta$ is a simplex, or $\Delta$ contains a vertex $v$ such that
		\begin{enumerate}
			\item ${\rm del}_\Delta (v)$ and ${\rm lk}_\Delta (v)$ are both vertex decomposable, and
			\item every facet of  ${\rm del}_\Delta (v)$ is a facet of $\Delta$.
		\end{enumerate}
	\item[$($b$)$]  \emph{weakly polymatroidal} if there is an ordering of the vertices of $\Delta$ so that, for any $F=\{a_1,\dots,a_c\}$ and $G=\{b_1,\dots,b_c\}$ in $\mathcal{F}(\Delta^*)$ with $a_1=b_1, \dots, a_{q-1}=b_{q-1}$ and $a_q<b_q$, there is a $p\geq q$ so that $G-\{b_p\} \cup \{a_q\} \in \mathcal{F}(\Delta^*)$. This condition is equivalent to $\left(I_{\Delta}\right)^{\vee}$ being a weakly polymatroidal ideal in the sense of \cite{KH06}. 
	\end{itemize}
\end{defn}

Let $\Gamma$ be a simplicial complex. It is well--known that $\Gamma$ is matroid $\Lra$ $\Gamma^{\vee}$ is weakly polymatroidal $\Lra$ $\Gamma$ is vertex decomposable $\Lra$ $\Gamma$ is shellable $\Lra$ $\Gamma$ is Cohen--Macaulay over any field.

As an example, the simplicial complex $\Delta$ on $[6]$ with $\F(\Delta^{\vee})=\{123, 125, 126, 146, 234, 245, 456 \}$ is not a matroid, but $\Delta^{\vee}$ is weakly polymatroidal.

\subsection{Polarizations}
\begin{defn}
	Suppose $M=\prod^n_{i=1} x_i^{b_i}$ is a monomial in $R$. The \textit{(standard) polarization} of $M$ is the monomial $M^{\pol}:=\prod^n_{i=1} \prod^{a_i}_{j=1} x_{i,j}$ inside the polynomial ring $k[x_{i,j} \mid 1 \le i \le n, 1 \le j \le a_i]$. If $J=(M_1,\ldots,M_r)$ is a monomial ideal in $R$, then its \textit{(standard) polarization} is the ideal $J^{\pol}:=(M_1^{\pol},\dots,M_r^{\pol})$ inside the polynomial ring $T:=k[x_{i,j} \mid x_{i,j} \in \supp\,M_h^{\pol} \mbox{ for some } h]$. 
\end{defn}

As shown in \cite{Fr88} (see also \cite[Lemma~4.2.16]{BH93} or \cite[Thm~21.10]{Pe11}), there is a tight relation between $I^{\pol}$ and $I$ (in fact, $I^{\pol}\subseteq T$ is a so--called {\em deformation} of $I\subseteq R$), and so $I^{\pol}$ captures a great deal of information about $I$ itself. We note the information relevant to our needs below. 

\begin{prop}\label{polcm}
Let $I$ be a monomial ideal. 
\begin{enumerate}
\item[$(1)$] $I$ is a complete intersection $($of grade $c$$)$ if and only if $I^{\pol}$ is. 
\item[$(2)$] $R/I$ is Cohen-Macaulay if and only if $T/I^{\pol}$ is Cohen-Macaulay.
\item[$(3)$] If $R/I$ satisfies $(S_{\ell})$, then so does $T/I^{\pol}$.
\item[$(4)$] $R/I$ is level if and only if $T/I^{\pol}$ is level.
\end{enumerate}
\end{prop}

Parts (1), (2) and (4) follow, for instance, from \cite[Thm~21.10(2)]{Pe11}. Part (3) follows, for instance, from \cite[Proof of Theorem 4.1]{MT09}.\medskip

Polarization allows one to study an arbitrary monomial ideal through the lens of simplicial complexes, whose rich combinatorics one can, in theory, exploit. The primary objects we wish to polarize in this work are the ideals $J=I_\Delta^{(m)}$.   
\begin{Notation}
	If $\Delta$ is any simplicial complex on $[n]$, we write $\Delta^{(m)}$ for the Stanley-Reisner complex $\Delta([I_{\Delta}^{(m)}]^{\pol})$ on $[n]\times [m]$ associated to $[I_{\Delta}^{(m)}]^{\pol}$.
\end{Notation}
Clearly, $\Delta^{(m)}$ can be viewed as a family of combinatorial invariants of $\Delta$.   As the simplicial complexes $\Delta^{(m)}$ are quite large, we study the structure of the much smaller simplicial complex $(\Delta^{(m)})^*$. In general, given a monomial ideal $J$, a first question to understand $(\Delta(J^{\pol}))^*$ would be: how do we identify its facets? 
Since they can be identified with the primes in $\Ass_T(k\left[(\Delta(J^{\pol}))\right])$, the answer is known when $J=I_{\Delta}^{(m)}$, see part (2) of the following result.
\begin{prop}[{\cite[Prop.~2.5]{Fa06}}]\label{polprimes}
\	 
\begin{enumerate}

\item[$(1)$] If $I=(x_{i_1}^{a_1},\dots,x_{i_r}^{a_r})$ then $\Ass_T(T/I^{\pol})=\{(x_{i_1,c_1},\dots,x_{i_r,c_r})\, \mid\, 1\leq c_i \le a_i\  \forall i\}$
 
\item [$(2)$] ${\rm Ass}_Tk\left[\Delta^{(m)}\right]$ consists  of the primes $\p:=(x_{i_1,a_1},x_{i_2,a_2},\dots,x_{i_c,a_c})$ such that
\begin{itemize}
	\item the prime $\ovl{\p}:=(x_{i_1},\dots,x_{i_c})$ is in $\Ass_R{k[\Delta]}$, 
	\item $1 \le a_i \le m$ for every $i$, and
	\item $\ssi{\p}:=\sum_{i=1}^c a_i$ is $\le c+m-1$. (see also Definition \ref{J_m}.)
\end{itemize} 
	\noindent Equivalently, $((i_1,a_1),\ldots,(i_c,a_c))$ is a facet of $[\Delta^{(m)}]^*$ if and only if $(i_1,\ldots,i_c)\in \Delta^*$, $1\leq a_i\le m$ for all $i$, and $\sum a_i \leq c+m-1$.	
\end{enumerate}
 
\end{prop}

\begin{example}
Let $\Delta=\langle 12, 23, 34 \rangle$, then $\Delta^*=\langle 12, 14, 34 \rangle$ and $\Ass_R(k[\Delta])=\{(x_1,x_2), (x_1,x_4), (x_3,x_4)\}$. 
Then
$$\begin{array}{ll}
	[\Delta^{(2)}]^*= &\langle ((1,1), (2,1)),\; ((1,2), (2,1)),\; ((1,1), (2,2)),\; ((1,1), (4,1)),\;((1,2), (4,1)),\;((1,1), (4,2)),\;\\
	& ((3,1), (4,1)),\;((3,2), (4,1)),\;((3,1), (4,2))\rangle ,\\
	\end{array}$$ or equivalently, 
$$\begin{array}{ll}
\Ass_T(k[\Delta^{(2)}])= &\left\{ (x_{11}, x_{21}),\; (x_{12}, x_{21}),\; (x_{11}, x_{22}),\; (x_{11}, x_{41}),\;(x_{12}, x_{41}),\; (x_{11}, x_{42}),\right.\\
& \left. (x_{31}, x_{41}),\;(x_{32}, x_{41}),\; (x_{31}, x_{42})\right\}.
	\end{array}$$
If we set $F=\{(1,2), (4,1)\}$ and $\p:=(x_{12}, x_{41})$, then $\bar{F}=\{1,4\}$,  $\bar{\p}=(x_1,x_4)$, and $\ssi{\p}=2+1=3$.
\end{example}

We record the following observation for future uses.

\begin{remark}\label{variables}
	The polarization of an ideal $J$ only affects the variables in the support of at least one minimal monomial generator $J$. So if a variable $z\in S$ does not appear in any minimal generator of $J$, then $z_1$ is the only variable associated to $z$ in the polarization, and it is a cone point of the simplicial complex associated to $J^{\pol}$.
	
	Since these variables/cone points do not affect the combinatorial and algebraic properties under consideration in this paper (being matroid, shellable, pure, Cohen--Macaulay, complete intersection, etc.),  in our arguments we will assume there are no cone points or, equivalently, that every variable divides at least one minimal monomial generator of $I_{\Delta}$.
	
In a similar spirit, in most statements we may assume there are no vertices $v$ in $[n]$ with $v\notin \Delta$, because those statements are true if and only if they are true for $\Delta|_{[n]-\{v\}}$.

\end{remark}

\subsection{Hochster--Huneke graphs}
We conclude this section with the definition of the Hochster--Huneke graph which is a graph associated to any Noetherian ring $S$. Historically, a more general version of this graph was first introduced by Hartshorne in \cite{Ha62} to establish connectedness results for locally Noetherian preschemes; see for instance \cite[Cor.~2.4]{Ha62}. Hochster and Huneke, in generalizing results of Faltings,  proved that, if $S$ is a complete equidimensional Noetherian local ring, then the connectedness of its Hochster--Huneke graph is equivalent to a number of very different--looking and homologically relevant properties of $S$ (e.g. the indecomposability of its canonical module) \cite[Thm~3.6]{HH94}. The name ``Hochster--Huneke graph" was first used in \cite{Zh07}. In the literature this graph may also be referred to as the ``dual graph".

In our context, we will need the labelled version of this graph employed, for instance, in \cite{Ho18}.

\begin{defn}
	Let $\Delta$ be a pure simplicial complex. The (labelled) {\em Hochster-Huneke graph} (or 	{\em dual graph}) of $k[\Delta]$ is the labelled graph $\mathbf{\mathcal{G}}(k[\Delta])$ with 
\begin{itemize}
	\item vertex set $\{v_\p\,\mid\,\p\in \Ass(k[\Delta])\}$, so in bijection with  $\Ass(k[\Delta])$,
	\item labeling given by $v_\p:=\{x_i\,\mid\, x_i \notin \p\}$ (so $v_\p$ is labeled by the variables {\em not} in $\p$), and
	\item the following edges: $\{v_\p, v_\q\}$ is an edge if and only if $|v_\p\cap v_\q|=|v_\p|-1$.
\end{itemize}
\end{defn}

\begin{remark}
One can check that $\mathbf{\mathcal{G}}(k[\Delta])$ is a relabeling of the facet-ridge graph of $\Delta^*$ or, for matroids, of the {\em matroid basis graph} of $\Delta^*$. (see e.g. \cite[Def.~1.2]{Ma73}.)
\end{remark}

We will then need a more sophisticated notion of ``connectedness" for labelled Hochster--Huneke graphs. It was introduced in \cite{Ho18}, inspired by work of several authors studying diameters of polyhedra and abstractions of polytopes.

\begin{defn}
The graph $\mathbf{\mathcal{G}}(k[\Delta])$ is {\em locally connected} if, for any two vertices $v_\p$, $v_\q$ of $\mathbf{\mathcal{G}}(k[\Delta])$,  there is a {\em locally connected path} between $v_\p$ and $v_\q$, i.e. a path $v_\p, v_{\p_1},\ldots, v_{\p_{r-1}}, v_\q$ in $\mathbf{\mathcal{G}}(k[\Delta])$ where each $v_{\p_i}$ in the path contains the set $v_\p \cap v_\q$.

If there is a locally connected path between $v_\p$ and $v_\q$, we write $=a(\p, \q)$ for the length of the shortest locally connected path connecting them.  
\end{defn}

\begin{remark}
From an algebraic perspective,
 \begin{itemize}
	\item $\{v_\p,v_\q\}$ is an edge of $\mathbf{\mathcal{G}}(I_\Delta$) if and only if ${\rm grade}(\p+\q)=1$ in $k[\Delta]$;
	\item $v_\p=v_{\p_0}, v_{\p_1},\ldots, v_{\p_{r-1}}, v_{\p_{r}}=v_\q$ is a locally connected path between $v_\p$ and $v_\q$ if and only if, for every $i$, ${\rm grade}(\p_i+\p_{i+1})=1$ and $\p_i\subseteq\p+\q$.
\end{itemize}
Also, we will often identify the vertex set of $\mathbf{\mathcal{G}}(k[\Delta])$ with $\Ass(k[\Delta])$, so we will say, for instance, that two associated primes $\p,\q$ are vertices of $\mathbf{\mathcal{G}}(k[\Delta])$,  and we will write $\p,\p_1,\ldots,\p_{r-1},\q$ for the locally connected path $v_\p, v_{\p_1},\ldots, v_{\p_{r-1}}, v_\q$.
\end{remark}

In general there can be multiple locally connected paths between two vertices. E.g. consider the complete intersection $I_{\Delta}=(ac,bd)= (a,b)\cap (b,c) \cap (c,d) \cap (a,d)$. Then $(a,b), (b,c), (c,d)$ and $(a,b), (a,d), (c,d)$ are two locally connected paths connecting the vertices $(a,b)$ and $(c,d)$. 

Our main interest in Hochster-Huneke graphs comes from the following:
\begin{theorem}[{\cite[Theorem 16]{Ho18}}]\label{lc}
	The following are equivalent:	
	\begin{enumerate}
		\item[$(1)$] $k[\Delta]$ satisfies Serre's condition $(S_2)$.			
		\item[$(2)$] $\Delta$ is pure and $\mathbf{\mathcal{G}}(k[\Delta])$ is locally connected. 
	\end{enumerate}
\end{theorem}

We illustrate the utility of Theorem \ref{lc} with the following preliminary result which will be used in the later sections. 

\begin{prop}\label{s2prelim}
Suppose $R/I_{\Delta}^{(m)}$ satisfies $(S_2)$ for some $m \ge 1$. Then $\Delta$ satisfies $(S_2)$. 
\end{prop}

Before proving the result, we establish a piece of notation that will be used throughout.
\begin{Notation}\label{J_m}
   Let $\Delta$ be a simplicial complex on $[n]$ and $m\in \ZZ_+$.
   \begin{itemize}
    \item We set $T=k[x_{ij}\,\mid\,1\leq i \leq n,\,1\leq j \leq m]$ and 
   $$
J_m=I_{\Delta^{(m)}} = [I_{\Delta}^{(m)}]^{\pol} \subseteq T.
   $$
  \item For any monomial prime $\p=(x_{i_1,a_1},\ldots,x_{i_c,a_c})$ in $T$, we define $\ssi{\p}$ as the {\ul{s}}um of the {\ul{s}}econd {\ul{i}}ndices of the variables in $\p$, i.e.
  $$
  \ssi{\p}=\sum_{i=1}^c a_i.
  $$
\item  We define the following total order on the variables of $T$:  
		$$
		x_{a,b}< x_{c,d}\; \Llra\;\left\{\begin{array}{l}
			b<d,\\
			\text { or }\\
			b=d \text{ and }a<c.
		\end{array}
		\right. 
		$$ 

\item We order the elements of $G((J_m)^\vee)$ in increasing order using Lex. (see also Notation \ref{order}.) Also, for any $M\in G((J_m^\vee))$, we set $C_M$ to be the colon ideal 
$$C_M:=(N\in G((J_m)^\vee)\,\mid\, N<M)\;:\;M.$$

   \end{itemize}
\end{Notation}
While the above orders on the variables of $T$ and on $G((J_m)^\vee)$ are {\em not} used in the proof of Proposition \ref{s2prelim}, they are actually crucial for the proofs of {\em all} the main results in this paper.

\begin{proof}

Suppose that $R/I_{\Delta}^{(m)}$ satisfies $(S_2)$ for some $m\geq 3$. Then $I_{\Delta}^{(m)}$ is a unmixed ideal, so $\Delta$ is pure. By Theorem \ref{lc}, we have that $\mathbf{\mathcal{G}}(J_m)$ is locally connected where $J_m\subseteq T$ is as in Notation \ref{J_m}. We prove that $\mathbf{\mathcal{G}}(k[\Delta])$ is locally connected. Indeed, let $\p,\q \in \Ass(R/I_{\Delta})$ and write $\p=(x_1,\dots,x_c)$ and $\q=(y_1,\dots,y_c)$. Proposition \ref{polprimes} yields that $\p'=(x_{1,1},x_{2,1},\dots,x_{c,1})$ and $\q'=(y_{1,1},y_{2,1},\dots,y_{c,1})$ are associated primes of $T/J_m=k[\Delta^{(m)}]$. As $\mathbf{\mathcal{G}}(J_m)$ is locally connected, there is a collection of $\p_i \in \Ass(k[\Delta^{(m)}])$ forming a locally connected path between $\p'$ and $\q'$. The collection $\bar{\p_i}$ forms a locally connected path between $\p$ and $\q$ (possibly of non--minimal length), proving $\mathbf{\mathcal{G}}(k[\Delta])$ is locally connected so that $\Delta$ satisfies $(S_2)$.
\end{proof}

The following remarks, except possibly for part (6)(b), follow immediately from the definition of locally connected path.
\begin{remark}\label{localpath}
Let $\p,\q\in \Ass(k[\Delta])$, and assume there is a locally connected path between them in $\mathbf{\mathcal{G}}:=\mathbf{\mathcal{G}}(k[\Delta])$. (E.g. if $\mathbf{\mathcal{G}}$ is locally connected.)
\begin{enumerate}
	\item If $c={\rm grade}(I_\Delta)$ in $R$ and  $r={\rm grade}(\p+\q)$ in $k[\Delta]$, then $\p\cap \q$ contains precisely $c-r$ variables;
	\item $a(\p,\q)\geq {\rm grade}(\p+\q)$;
	\item $a(\p,\q)=1$ if and only if ${\rm grade}(\p+\q)=1$;
	\item if $a(\p,\q)=2$, then ${\rm grade}(\p,\q)=2$;
	\item if $a(\p,\q)={\rm grade}(\p+\q)$ then every prime in any locally connected path of length $a(\p,\q)$ between $\p$ and $\q$ contains every variable in $\p\cap \q$;
	\item if $\p=:\p_0,\p_1,\ldots,\p_r,\q=:\p_{r+1}$ is a locally connected path, then
	\begin{enumerate}
		\item ${\rm grade}(\p+\p_i)\leq i$ in $k[\Delta]$ for every $i$;
		\item $a(\p, \p_2)\leq 2$, and, $a(\p,\p_2)=2$ $\Llra$ ${\rm grade}(\p+\p_2)=2$.
	\end{enumerate} 
\end{enumerate}
\end{remark}

\begin{proof}
(6)(b) 	By part (a) ${\rm grade}(\p+\p_2)\leq 2$. Now, the forward implication is part (4).  
For the other implication, it suffices to prove that $\p_1\subseteq \p+\p_2$, because then $\p,\p_1,\p_2$ is a locally connected path. Write $\p_1=\p-(z)+(w)$ for some variables, $z\in \p$ and $w\in \q$. If $w\in \p_2$ we are done. If $w\notin \p_2$, then $\p_2=\p_1-(w) + (u)$ for some variable $u\in \p+\q$. It follows that $\p_2=\p_1-(z)+(u)$, thus ${\rm grade}(\p+\p_2)=1$, yielding a contradiction. 
\end{proof}

Example \ref{BadExamples} illustrates some of the subtleties one has to be aware of when employing the Hochster--Huneke graph in a proof, e.g. in the proof of Theorem \ref{cmsymb}. In part (1) of the example we show that the inequality of Remark \ref{localpath}(2) may be strict if ${\rm grade}(\p+\q)\geq 2$, and the converse of (4) does not hold. 
In part (2) we show that Remark \ref{localpath}(6)(b) is true {\em only} for $\p_2$ and not, in general, for  any $\p_i$ with $i\geq 3$. In fact, even if $\p,\p_1,\ldots,\p_r,\q$ is a {\em shortest} path connecting $\p$ and $\q$, then $a(\p,\p_i)$ could be $>i$, if $i\geq 3$. For sake of simplicity, here we only show an example where $a(\p,\p_3)>3$. Similar examples can be constructed where $a(\p,\p_i)>i$ for any $i\geq 3$.

\begin{example}\label{BadExamples}
(1) Let $I_{\Delta}=(a,b,c)\cap (a,b,x)\cap (b,x,y)\cap (a,x,y)\subseteq k[a,b,c,x,y]$. Letting $\p:=(a,b,c)$ and $\q:=(a,x,y)$, then $a(\p,\q)=3>{\rm grade}(\p+\q)=2$.\\
(2) Let $I_\Delta$ be the following ideal
$$	
\begin{array}{ll}
I_\Delta = & \quad	(a,b,c,d) \cap (a,b,c,x)\cap (a,b,x,y) \cap (a,b,y,z)\cap (a,b,z,w)\\
& \cap\; (a,x,z,w) \cap (x,y,z,w)\cap (a,d,y,z) \cap (c,d,y, z) \cap (b,c,d,z).	
\end{array}
$$	
Then $\p:=(a,b,c,d)$, $\p_1:=(a,b,c,x)$, $\p_2:=(a,b,x,y)$, $\p_3:=(a,b,y,z)$, $\p_4:=(a,b,z,w)$, $\p_5:=(a,x,z,w)$, $\q:=(x,y,z,w)$ is the shortest locally connected path between $\p$ and $\q$, and $a(\p,\p_3)=4$, because $x\notin \p+\p_3$, so $\p$, $(b,c,d,z)$, $(c,d,y, z)$, $(a,d,y,z)$, $\p_3$ is the shortest path between $\p$ and $\p_3$.

\end{example}

\section{The Cohen-Macaulay Property and Serre's Condition for Symbolic Powers}\label{squarefreesection}

The first main result of this section provides a short proof of Theorem \ref{cmsymbintro} (see Theorem \ref{cmsymb} below). We then illustrate how the methods of the proof bear more fruit and new results.

\subsection{The proof of Theorem \ref{cmsymbintro}} 
\begin{theorem}\label{cmsymb}(Thm.~\ref{cmsymbintro} of the introduction)
	Let $\Delta$ be a simplicial complex. TFAE:
	\begin{enumerate}
		\item[$(1)$] $\Delta$ is a matroid.		
		\item[$(2)$] The simplicial complex $\Delta^{(m)}$ is vertex decomposable for every $m \ge 1$.		
		\item[$(3)$] $R/I_{\Delta}^{(m)}$ is Cohen-Macaulay over any field, for every $m \ge 1$.		
		\item[$(4)$] $R/I_{\Delta}^{(m)}$ satisfies $(S_2)$, for every field $k$ and every $m \ge 1$.		
		\item[$(5)$] $R/I_{\Delta}^{(m)}$ satisfies $(S_2)$, for some field $k$ and some $m \ge 3$.		
	\end{enumerate}	
\end{theorem}

\begin{proof}
	The implications $(3) \Lra (4) \Lra (5)$ are obvious. $(2) \Lra (3)$ The  assumptions force $\Delta^{(m)}$ to be pure shellable, so it is Cohen-Macaulay over every field. Now use Proposition \ref{polcm}(2).
 
	$(5) \Lra (1)$. By Remark \ref{variables} we may assume both $\Delta$ and $\Delta^*$ have no cone points, so $d=\dim(\Delta)\leq n-2$, i.e. $c:=n-d$ is at least 2. 
	First, $\Delta$ satisfies $(S_2)$ by Proposition \ref{s2prelim}, in particular $\Delta$ is pure. 
	To prove $\Delta$ is a matroid, it suffices to show that $\Delta^*$, or equivalently $\Ass(k[\Delta])$, is a matroid. Let $\p=(x_1,\ldots,x_c)$ and $\q=(y_1,\ldots,y_c)$ be distinct primes in $\Ass(k[\Delta])$. Say $x_1\notin \q$. Without loss of generality, it suffices to show that there exists $i$ such that $\q_1:=(y_1,\dots,\hat{y}_i,\dots,y_c)+(x_1) \in \Ass(k[\Delta])$. Equivalently, we need to prove the existence of a prime $\q_1\in \Ass(k[\Delta])$ with ${\rm grade}(\q_1+\q)=1$ in $k[\Delta]$ and $x_1\in \q_1$.

Since $\Delta$ is $(S_2)$, then by Theorem \ref{lc} there is a locally connected path between $\p$ and $\q$. We prove the existence of $\q_1$ by induction on ${\rm grade}(\p+\q)\geq 1$ in $k[\Delta]$. 
	When ${\rm grade}(\p+\q)=1$ there is nothing to prove. If ${\rm grade}(\p+\q)= 2$ in $k[\Delta]$, then we can write $\p=(x_1,x_2,z_3,\dots,z_{c})$ and $\q=(y_1,y_2,z_3,\dots,z_{c})$ for distinct variables $x_1,x_2,y_1,y_2,z_1,\ldots,z_c$. 
	
Consider the primes $P:=(x_{1,1},x_{2,m},z_{3,1},\dots,z_{c,1})$ and $Q:=(y_{1,2},y_{2,m-1},z_{3,1},\dots,z_{c,1})$. Since $\ovl{P}=\p$, $\ovl{Q}=\q$, and $\ssi{P}=1+m+1+\ldots +1=c+m-1=\ssi{Q}$, then, by Proposition \ref{polprimes}, $P,Q\in \Ass(k[\Delta^{(m)}])$.
	We claim that there is a locally connected path $Q,Q_1,P$ of length 2 between $Q$ and $P$ and such that $x_{1,1}\in Q_1$.  This claim implies the base case by setting $\q_1:=\bar{Q_1}$, because $\q,\ovl{Q_1},\p_1$ is a locally connected path in $\mathbf{\mathcal{G}}(k[\Delta])$ and  $x_1\in \ovl{Q_1}$. 
	
	Let $Q=:P_0,P_1,\ldots,P_r,P_{r+1}:=P$ be a shortest locally connected path between $Q$ and $P$, so $P_j\subseteq P+Q$ for all $j$.
	Let $1\leq u \leq r$ be such that $P_u$ does not contain both $y_{1,2}$ and $y_{2,m-1}$
	(such $u$ exists because $Q=P_0$ contains both of them, while $P_{r+1}$ does not contain any of them). We show that $x_{2,m}\notin P_u$. Indeed, if $x_{2,m}\in P_u$ then the inequality $\ssi{P_u}\leq c+m-1$ implies that the second index of the other $c-1$ variables in $P_u$ is $1$. Since $m\geq 3$, then the second index of $y_{2,m-1}$ is at least 2, thus $x_{1,1},z_{3,1},\ldots,z_{c,1}$ are the only variables with second index 1 in $P+Q$. It follows that $P_u=P$, thus $Q=P_0,P_1,\ldots,P_u=P$ is a locally connected path, contradicting the minimal length of the previous path.
	
	Now, since $P_u\subseteq P+Q$ but $x_{2,m}\notin P_u$, and one among $y_{1,2}$ and $y_{2,m-1}$ is not in $P_u$, then $P_u$ is generated by $x_{1,1},z_{3,1},\ldots,x_{c,1}$, and one between $y_{1,2}$ and $y_{2,m-1}$. In either case, $Q,P_u,P$ is a locally connected path and $x_{1,1}\in P_u$. This proves the base case ${\rm grade}(\p+\q)=2$.
	
	If ${\rm grade}(\p+\q) \ge 3$, let $\p=\p_0,\p_1,\ldots, \p_{a-1}, \p_a=\q$ be a shortest locally connected path between $\p$ and $\q$. By Remark \ref{localpath}(2) and (6)(b) we know that $a\geq 3$ and $a(\p,\p_2)\leq 2$. However if $a(\p,\p_2)=1$, then $\p,\p_2$ is a locally connected path, so $\p,\p_2,\p_3,\ldots,\p_a=\q$ is a locally connected path of length $<a$, yielding a contradiction. So $a(\p,\p_2)=2$ and thus ${\rm grade}(\p +\p_2)=2$, by Remark \ref{localpath}(6)(b). 
    
	Recall that $x_1\in \p$. Next, we observe we may assume $x_1\in \p_1$. Indeed, if $x_1\in \p_2$, since $a(\p,\p_2)={\rm grade}(\p+\p_2)=2$, then $x_1\in \p_1$ by Remark \ref{localpath}(5). If, instead, $x_1\notin \p_2$, then by the base case there exists a locally connected path $\p,\p_1',\p_2$ with $x_1\in \p_1'$. Since $\p_1'\subseteq \p+\p_2\subseteq \p+\q$, then $\p,\p_1',\p_2,\ldots,\p_{a-1},\q$ is a locally connected path of minimal length connecting $\p$ and $\q$, and so, after replacing $\p_1$ by $\p_1'$, we may assume $x_1\in \p_1$. The statement now  follows by inductive hypothesis, since $x_1\in \p_1$ and ${\rm grade}(\p_1+\q)\leq a-1$, by Remark \ref{localpath}(6)(a).

$(1) \Lra (2)$. Let $J_m=[I^{(m)}]^{\pol}\subseteq T$ be as in Notation \ref{J_m}. For ease of notation, we identify the variables of $T$ with the vertex set of $\Delta^{(m)}$, and the variables of $R$ with the vertex set of $\Delta$. Since $\Delta$ is a pure simplicial complex, so is  $\Delta^{(m)}$, and thus the ideal $J_m$ is unmixed. To prove vertex decomposability of $\Delta^{(m)}$ we show that its Alexander dual $[\Delta^{(m)}]^{\vee}$ is weakly polymatroidal {\em under the order of the variables defined in Notation \ref{J_m}}. For the rest of the proof, we write every facet of $[\Delta^{(m)}]^{\vee}$ following this order, i.e. if we write $F=\{x_{i_1,a_1},x_{i_2,a_2},\ldots, x_{i_s,a_s}\}$ we silently mean $x_{i_1,a_1}<x_{i_2,a_2}< \ldots < x_{i_s,a_s}$, in particular, $a_1\leq a_2\leq \ldots \leq a_s$. To $F$ we associate a prime ideal $\overline{\p_{F^*}}\subseteq k[\Delta]$, obtained by ``forgetting" the second indices in $F$, and an integer $\ssi{F}$ as follows 
$$\ovl{\p_{F^*}}:=(x_{i_1},\ldots,x_{i_s})\subseteq S,\qquad \quad \text{ and }\quad \qquad \ssi{F}:=\sum_{\ell=1}^s a_\ell.$$

Finally, we recall that
		$$
		\F([\Delta^{(m)}]^{\vee})=\{F:=\{x_{i_1,a_1},x_{i_2,a_2},\ldots, x_{i_c,a_c}\} \,\mid \, \ovl{\p_{F^*}}\in \Ass(k[\Delta]),\, a_i\in \ZZ_+ \text{ and } \ssi{F}\leq m+c-1\},
		$$
  and we totally order $\F([\Delta^{(m)}]^{\vee})$ using the Lex order of Notation \ref{J_m}. %

Let $F<G\in \mathcal \F([\Delta^{(m)}]^{\vee})$. We can write them as $G=\{x_{j_1,b_1},x_{j_2,b_2},\ldots, x_{j_c,b_c}\}$ and $F=\{x_{j_1,b_1}, x_{j_2,b_2},\ldots, x_{j_r,b_r},\,x_{i_{r+1},a_{r+1}},\ldots, x_{i_c,a_c}\}$ for some $r\geq 0$, with $x_{i_{r+1},a_{r+1}}< x_{j_{r+1},b_{r+1}}$. Notice that in particular $a_{r+1}\leq b_{r+1}$. We distinguish two cases:
\begin{itemize}
\item \underline{Case 1: $i_{r+1} \in \{j_{r+1},\ldots,j_c\}$.} Let $t\geq r+1$ be such that $j_t=i_{r+1}$. Set $G_0:=G \cup\{x_{i_{r+1},a_{r+1}}\} -\{x_{i_{r+1}, b_t}\}$. We only need to check that $G_0\in \F([\Delta^{(m)}]^{\vee})$. First, $\ovl{\p_{G_0^*}}=(x_{i_1},\ldots,x_{i_r},x_{j_{r+1}},\ldots,x_{j_c})=\ovl{\p_{G^*}}$, and since $G\in \F([\Delta^{(m)}]^{\vee})$, then $\ovl{\p_{G_0^*}}=\ovl{\p_{G^*}}\in \Ass(k[\Delta])$.
Secondly, since $a_{r+1}\leq b_{r+1}\leq b_t$, then
$$\ssi{G_0} = \sum_{\ell\neq t} b_\ell + a_{(r+1)} \leq    \sum_{\ell\neq t} b_\ell + b_{t} =: \ssi{G}\leq m+c-1,$$
where the rightmost inequality holds because $G \in \F([\Delta^{(m)}]^{\vee})$. Thus $G_0 \in \F([\Delta^{(m)}]^{\vee})$.\\
\item \underline{Case 2: $i_{r+1}\notin \{j_{r+1},\ldots,j_c\}$.} Since  $\ovl{\p_{F^*}},\ovl{\p_{G^*}}\in \Ass(k[\Delta])$, $\Delta$ is a matroid, and $x_{i_{r+1}}\in \ovl{\p_{F^*}} - \ovl{\p_{G^*}}$, then there exists  $x_{j_t}\in \ovl{\p_{G^*}} - \ovl{\p_{F^*}}$ such that replacing in $\ovl{\p_{F^*}}$ the variable $x_{i_{r+1}}$ with $x_{j_t}$ gives another associated prime, $\q$, of $k[\Delta]$. Observe that since $F<G$ and $x_{j_t}\notin \ovl{\p_{F^*}}$, then $t\geq r+1$. Now, let $G_0:=G \cup\{x_{i_{r+1},a_{r+1}}\} -\{x_{j_t, b_t}\}$. Now, $\ssi{G_0}\leq \ssi{G}\leq m+c-1$ as above, moreover $\ovl{\p_{G_0^*}}=\q\in \Ass(k[\Delta])$, thus $G_0\in \F([\Delta^{(m)}]^{\vee})$.	
\end{itemize}	
		
Then $[\Delta^{(m)}]^{\vee}$ is weakly polymatroidal, so $\Delta^{(m)}$ is vertex decomposable (e.g. \cite[Thm~2.5]{Mo11}). 

\end{proof}

\begin{Remark}\label{m=3}
The assumption that $m\geq 3$, and the choices of $P$ and $Q$, are both crucial in the proof of $(5)\Lra (1)$ as one can see, for instance, in the base case of the induction. In fact, that part of the proof needs both $y_{1,2}$ and $y_{2,m-1}$ to have second index $>1$ so that, by maximality of $\ssi{Q}$, $x_{2,m}$ cannot replace any of them. These second indices are both $>1$ because $m\geq 3$. So our proof shows very clearly  why $m\geq 3$ is needed.

If $m=2$, i.e. if $R/I_{\Delta}^{(2)}$ is $(S_2)$, then it is not necessarily true that $\Delta$ is a matroid. See for instance Proposition \ref{2ndSymbPwr} and the discussion before Definition \ref{2-loc}.
\end{Remark}

In the rest of this section we provide further illustration of the benefits of our approach.

\subsection{Application 1. The simplicial complexes $\Delta^{(m)}$ are glicci} 
Before stating our first consequence of Theorem \ref{cmsymb}, we provide some context. Although liaison theory had been used since the late nineteenth century, it was first formally introduced by Peskine and Szpiro in \cite{PS74}. Schenzel and, later, Nagel \cite{Sc82} \cite{Na98} showed that several useful theorems still apply in the context of the more general theory of $G$--liaison. Possibly, the single most relevant question in G--liaison asks whether every Cohen--Macaulay ideal $J_1$  in a polynomial ring over a field is {\em glicci}, i.e. if there exist $s\in \ZZ_+$ and, Gorenstein ideals $G_1,\ldots,G_{s-1}$ with $G_1\subseteq J_{1}$ and  $G_i\subseteq J_i:=G_{i-1}:J_{i-1}$ for $i\geq 2$, such that $J_{s}:=G_{s-1}:J_{s-1}$ is a complete intersection ideal \cite[Question~1.6]{KMMN01}. 

A Cohen--Macaulay ideal $J_0$ is said to be {\em licci}, if, additionally, one assumes that $G_1,\ldots,G_{s-1}$ are complete intersections ideals. Being licci is much more restrictive than being glicci, and the next result provides yet another illustration of it. 
Following the principle stated at the beginning of Section 2, we say that a simplicial complex $\Gamma$ is {\em licci} ({resp. \rm glicci}) if $I_\Gamma$ is licci (resp. glicci).

 An immediate consequence of our proof of Theorem \ref{cmsymb} is that if $\Delta$ is any matroid, then $\Delta^{(m)}$ is glicci for every $m$. This new result follows from our addition to Theorem \ref{cmsymbintro}, i.e. part (2). To the best of our understanding, it does follow (at least, not quickly) from the previous proofs of Theorem \ref{cmsymbintro} in the literature, thus it illustrates some of the benefits of our approach.  

\begin{Corollary}\label{glicci}
Let $\Delta$ be a matroid of corank $c$, and let $m\geq 2$ be an integer. 

\begin{enumerate}
    \item If $c\leq 2$, then $\Delta^{(m)}$ is a licci simplicial complex and, thus, glicci. The ideals $I_{\Delta}^{(m)}$ are all licci and glicci.
    \item If $c\geq 3$, then the simplicial complex $\Delta^{(m)}$ is not licci, and $I_{\Delta}^{(m)}$ is not a licci ideal. The simplicial complex $\Delta^{(m)}$ is glicci.  
\end{enumerate}
\end{Corollary}

\begin{proof}
	First of all, being licci is preserved by specializations and deformations (e.g. by \cite[Prop.~2.8 and Lem.~2.16]{HU87}). In particular, 
	$I_{\Delta}^{(m)}$ is a licci ideal if and only the simplicial complex $\Delta^{(m)}$ is licci. 
	
	(1) It is well--known (it goes back to Ap\'ery and Gaeta) that every Cohen--Macaulay ideal of grade $c\leq 2$ in $R$ is licci \cite{Ap45,Ga52}. The second part holds because every licci ideal is glicci.
	
	(2) The proof of Theorem \ref{cmsymb} shows that $\Delta^{(m)}$ is vertex decomposable, therefore, according to \cite[Thm~3.3]{NR08}, $\Delta^{(m)}$ is squarefree glicci (see \cite[Def.~2.1]{NR08}) and, thus, glicci. 
	
	On the other hand, it follows by \cite[Prop.~2.3 or Thm.~2.10]{PU98} that $I_\Delta^{(m)}$ is not licci so, by the above, $\Delta^{(m)}$ is not a licci simplicial complex.
\end{proof}

We actually conjecture that $I_{\Delta}^{(m)}$ is always glicci in this setting:
\begin{conjecture}\label{Conj-glicci}
$I_{\Delta}^{(m)}$ is glicci for every matroid $\Delta$ and every $m\in \ZZ_+$. 
\end{conjecture}

Unfortunately, in general, being glicci is neither preserved under specializations nor deformations. So Corollary \ref{glicci}(2) does {\em not} imply Conjecture \ref{Conj-glicci}. It is not known whether the stronger condition we actually use, i.e. that $\Delta^{(m)}$ is {\em squarefree} glicci, implies that $I_\Delta^{(m)}$ is glicci. For instance, 
\cite[Thm~3.10]{FKRS} cannot be applied because, in general, the depolarizations of the deletions appearing in the vertex decomposition are far from being generically Gorenstein if $c\geq 3$ and $m\geq 2$. (In fact, often times already the polarization of the first deletion is not generically Gorenstein.)

\subsection{Application 2. Mixed symbolic powers of $I_\Delta$} 
Using our methods, one can provide variations on Theorem \ref{cmsymb} allowing the study of mixed symbolic powers, i.e. ideals $I$ of the form  $I=\p_1^{m_1}\cap \ldots \cap \p_r^{m_r}$. In general,  except for trivial situations (e.g. if $I$ is principal or $\dim(R/I)\leq 1$), these ideals $I$ are much more complicated to investigate than the uniform case, where $m_1=m_2=\ldots=m_r$. For instance, even when $\p_1\cap \ldots \cap \p_r$ is the Stanley--Reisner ideal of a matroid, it is very complicated to completely characterize all the exponents $m_1,\ldots,m_r$ for which $R/I$ has the $(S_2)$ property. One of the rare instances where this characterization is known is when $\sqrt{I}$ is the Stanley--Reisner ideal of the uniform matroid of rank 2 on $\{a,b,c,d\}$, see \cite[Theorem]{Fr08}. Since in such case $\dim(R/I)=2$, then $R/I$ is $(S_2)$ if and only if it is Cohen--Macaulay, a fact used by Francisco to employ crucial additional tools. Nevertheless, the entire paper \cite{Fr08} is dedicated to finding an explicit numerical characterization, which is quite complicated and illustrates the difficulty in studying these ideals. 
\medskip

In the next two results we provide fairly strong combinatorial obstructions to mixed symbolic powers satisfying Serre's condition $(S_2)$. 
\begin{prop}\label{SuffMatroid}
	Let $\Delta$ be a pure simplicial complex, let $F_1,\ldots,F_r$ be its facets and $I_{\Delta}=\p_1\cap \ldots \cap \p_r$, where $\p_i:=\p_{F_i}$. 
	Let $m_1,\ldots,m_r$ be integers with $m_i\geq 3$ for all $i$ and $\max\{m_i\}\leq 2\min\{m_i\}-3$.
Let $H:=\p_1^{m_1}\cap \ldots \cap \p_r^{m_r}$. If $R/H$ satisfies Serre's condition $(S_2)$, then $\Delta$ is a matroid.
\end{prop}

\begin{proof}
Let $\p=(x_1,\ldots,x_c)$ and $\q=(y_1,\ldots,y_c)$. The proof of (5) $\Lra$ (1) of Theorem \ref{cmsymb} holds verbatim after making the following adjustments in the base case ${\rm grade}(\p +\q)=2$ of the induction. Let $\p^m$ and $\q^u$ be the $\p$-primary and $\q$-primary components of $H$, we still take $P:=(x_{1,1},x_{2,m},z_{3,1},\dots,z_{c,1})$ but this time we take $Q$ to be $(y_{1,b},y_{2,u+1-b},z_{3,1},\dots,z_{c,1})$ where both $b$ and $u+1-b$ are $\leq m-1$. Thus, if $m>u$, one can take any $1\leq b\leq u$. If $m\leq u$, take $b:=m-1$; notice that since $u\leq 2m-3$, then also $u+1-(m-1)$ is at most $m-1$, as needed.
\end{proof}

The condition on the $m_i$'s is sharp as the following example illustrates. 
\begin{example}
Let $\Delta=\langle ab, bc, bd, cd\rangle$, so $I_\Delta=(a,b)\cap (a,c)\cap (a,d)\cap (c,d)\subseteq k[a,b,c,d]=R$. 

Let $H=(a,b)^3\cap (a,c)^4\cap (a,d)^4\cap (c,d)^3$, so 
$\max\{m_i\}=4=2\min\{m_i\}-2$. By \cite[Theorem]{Fr08} the ring $R/H$ is Cohen--Macaulay and, thus, $(S_2)$, but $\Delta$ is not a matroid.
\end{example}

Our next result is an improved version of a theorem proved by Herzog, Takayama and Terai \cite[Thm.~3.2]{HTT05} and Minh and Trung \cite[Cor.~1.9]{MT11}. Specifically, we add to these results part (a), showing how we can use the $(S_2)$ property of a modification of $I_{\Delta}$ to deduce combinatorial information about $\Delta$, part (b).(3), which relaxes the Cohen--Macaulay property of the previously mentioned results, and part (b).(2), which gives concrete ideals to test the condition.

\begin{theorem}\label{mixedSymbPwrs}
	Let $\Delta$ be a pure simplicial complex on, let $F_1,\ldots,F_r$ be its facets and $I_{\Delta}=\p_1\cap \ldots \cap \p_r$, where $\p_i:=\p_{F_i}$. Then 
	\begin{itemize}
		\item[$($a$)$] If there exists $m\geq 3$ such that $R/(\p_1^m \cap \p_2 \cap \ldots \cap \p_r)$ satisfies Serre's condition $(S_2)$, then the labelled graph  $\mathbf{\mathcal{G}}(k[\Delta])$ is a cone over $v_{\p_1}$.

		\item[$($b$)$] The following are equivalent:
		\begin{enumerate}
			\item $\p_1^{m_1}\cap \ldots \cap \p_r^{m_r}$ is Cohen--Macaulay for every $m_1,\ldots,m_r\in \ZZ_+$;
			\item for every $i$ there exists $m_i\geq 3$ such that $R/(\p_1\cap \ldots \cap \p_{i-1} \cap \p_i^{m_i} \cap \p_{i+1}\cap \ldots \cap \p_r)$  satisfies Serre's condition $(S_2)$;
			\item $R/(\p_1 \cap \ldots \cap \p_{i-1} \cap \p_i^3 \cap \p_{i+1}\cap \ldots \cap \p_r)$  satisfies Serre's condition $(S_2)$ for all $1\leq i \leq r$;

			\item $\Delta$ is a cone over a 0-dimensional complex.
		\end{enumerate}
	\end{itemize}
\end{theorem}

\begin{proof}
(a) Let $c:={\rm grade}(I_\Delta)$, $H:=\p_1^m \cap \p_2 \cap \ldots \cap \p_r$ and let $J=H^{\pol}\subseteq T$. By Proposition \ref{polcm} (3), $T/J$ satisfies $(S_2)$. By definition of $H$, one has $\ssi{\q}=c$ for every $\q\in \Ass(T/J)$, except if $\bar{\q}=\p_1$.  

We need to show that $|F_1\cap F_i|=|F_1|-1$. Without loss of generality, for simplicity of notation, we prove $|F_1\cap F_2|=|F_1|-1$. Assume not, then $|F_1\cap F_2|<|F_1|-1$, so ${\rm grade}(\p_1+\p_2)\geq {\rm grade}(\p_1)+2$ and one can write $\p_1=(y_1,y_2)+P_1$ and $\p_2=(x_1,x_2)+P_2$, where $x_1,x_2 \notin \p_1$, $y_1,y_2\notin \p_2$, and $P_1,P_2$ are prime ideals of grade $c-2$. By Proposition \ref{polprimes}, the ideals $\q_1:=(y_{1,2}, y_{2,m-1})+(z_{j,1}\,\mid\,z_j\in P_1)$ and $\q_2:=(x_{1,1}, x_{2,1}) + (w_{j,1}\,\mid\,w_j\in P_2)$ are in $\Ass(T/J)$, so, by assumption and Theorem \ref{lc}, there is a locally connected path $\q_1,Q_1,\ldots,Q_s,\q_2$ in $\mathbf{\mathcal{G}}(T/J)$. 

Since $y_{1,2}$, $y_{2,m-1}$ are both in $\q_1$ and neither of them is in $\q_2$, in the locally connected path there is a smallest index $1\leq u \leq s$ such that $Q_u$ does not contain both $y_{1,2}$ and $y_{2,m-1}$. One then has $Q_{u-1}=(y_{1,2}, y_{2,m-1},\ldots)$ and $Q_u$ is obtained from $Q_{u-1}$ by replacing only one between $y_{1,2}$ or $y_{2,m-1}$ with a variable $t$ chosen among the variables $z_{j,1}$ or $w_{j,1}$ in $(\q_1+\q_2)-Q_{u-1}$. Since at least one of $y_1$ and $y_2$ is not contained in $\ovl{Q_u}$, then $\bar{Q_u}\neq \p_1$. On the other hand, since $m\geq 3$, then $Q_u$ contains one variable (either $x_{1,2}$ or $x_{2,m-1}$) with second index $\geq 2$, so in particular $\ssi{Q_u}>c$, therefore, by the above, $\bar{Q_u}=\p_1$, which yields a contradiction.

(b) The implications (1) $\Lra$ (2) $\Lra$ (3) are clear. (3) $\Lra$ (4) holds because by (a) for any two facets $F,G$ of $\Delta$ one has $|F\cap G|=|F|-1$. (4) $\Lra$ (1) Let $m_1,\ldots,m_r\in \ZZ_+$ and let $L:=\p_1^{m_1}\cap \ldots \cap \p_r^{m_r}$. By Remark \ref{variables} we may assume $\Delta$ is the simplicial complex consisting only of the $n$ vertices. Then  $\dim(R/\p_i)=1$ for every $i$, so $\dim(R/L)=1$, and since the maximal ideal is not an associated prime, then $R/L$ is Cohen--Macaulay.

\end{proof}

Theorem \ref{cmsymb} gives a combinatorial characterization of simplicial complexes $\Delta$ for which there exists $m\geq 3$ such that $R/I_{\Delta}^{(m)}$ has Serre's condition $(S_2)$. It is natural to ask a similar question for $R/I_{\Delta}^{(2)}$. In this regard, Rinaldo, Terai and Yoshida proved the following characterization: $R/I_{\Delta}^{(2)}$ is $(S_2)$ if and only if for every face $F\in \Delta$ for which $\dim\left(\lk_{\Delta}(F)\right)\geq 1$, the 1-skeleton of $\lk_{\Delta}(F)$ has diameter at most 2, \cite[Cor.~3.3]{RTY11}. 

Also, Minh and Trung proved in \cite[Thm.~2.5]{MT11} that if $\Delta$ is a tight simplicial complex, then $R/I_{\Delta}^{(2)}$ is Cohen--Macaulay. Recall that $\Delta$ is {\em tight}
if it is pure and for any two facets $F\neq G$ and any $i\in F-G$ and $j\in G-F$ there exists a facet $H$ such that $F\cap G \subseteq H \subseteq F\cup G$ and $H\cap \{i,j\}\neq \emptyset$.

Drawing inspiration from the notion of tight simplicial complexes, we provide here an alternative combinatorial characterization of the simplicial complexes for which $R/I_{\Delta}^{(2)}$ satisfies $(S_2)$.

\begin{defn}\label{2-loc}
We  say that a pure simplicial complex $\Delta$ {\em 2-locally connected} if for any  two facets $F\neq G$ and any $i\in F-G$ and $j\in G-F$ there is a sequence of facets $F=F_0,F_1,\ldots,F_r,F_{r+1}=G$ such that for every $0\leq h \leq r-1$ one has $|F_h\cap F_{h+1}|=|F_h|-1$, $F_h\subseteq F\cup G$ and $F_h\cap \{i,j\}\neq \emptyset$.
\end{defn}

\begin{prop}\label{2ndSymbPwr}
$R/I_{\Delta}^{(2)}$ is $(S_2)$ if and only if $\Delta$ is 2-locally connected.
\end{prop}

\begin{proof} 
By Proposition \ref{polprimes}, it suffices to prove that $k[\Delta^{(2)}]$ is $(S_2)$ if and only if $\Delta$ is 2-locally connected. \\
``$\Lra$" Clearly, $\Delta$ is 2-locally connected if and only if for any $\p\neq \q\in \Ass(k[\Delta])$ and any $x_i\in \q-\p$ and $x_j\in \p-\q$ there is a locally connected path between $\p$ and $\q$ such that no prime in the path contains both $x_i$ and $x_j$, so we prove this latter condition. Write $\p=(x_i,y_{1},\ldots,y_{c-1})$ and $\q=(x_j,z_1,\ldots,z_{c-1})$, let $P:=(x_{i,2},y_{1,1},\ldots,y_{c-1,1})$ and $Q:=(x_{j,2},z_{1,1},\ldots,z_{c-1,1})$; by Proposition \ref{polprimes} they are both in $\Ass(k[\Delta^{(2)}])$. Since $k[\Delta^{(2)}]$ is $(S_2)$, there is a locally connected path $P,P_1,\ldots,P_r,Q$. By Proposition \ref{polprimes}, $\ssi{P_h}$, i.e. the sum  of the second indices in $P_h$, is at most $c+1$ for every $P_h\in \Ass(k[\Delta^{(2)}])$, so none of the $P_h$'s contains both $x_{i,2}$ and $x_{j,2}$. Taking bars, we obtain a locally connected path $\ovl{P}=\p,\ovl{P_1},\ldots,\ovl{P_r},\ovl{Q}=\q$ in $\mathbf{\mathcal{G}}(\Delta)$ as in the above statement.

``$\Lla$" By the above equivalent characterization of 2-locally connected, it is clear that $\mathbf{\mathcal{G}}(\Delta)$ is locally connected. We prove  $\mathbf{\mathcal{G}}(\Delta^{(2)})$ is locally connected. Let $P\neq Q\in \Ass(k[\Delta^{(2)}])$, by Proposition \ref{polprimes} one has $\ssi{P}, \ssi{Q}\in \{c,c+1\}$, so either all variables in $P$ and $Q$ have second index 1, or all variables except one have second index 1, and the remaining variable has second index 2.
 
First, assume $\ovl{P}=\ovl{Q}$. If either $\ssi{P}=c$ or $\ssi{Q}=c$, then $P\cap Q$ contains $c-1$ variables, so $P,Q$ is a locally connected path. If $\ssi{P}=\ssi{Q}=c+1$, then $P\cap Q$ contains $c-2$ variables. Say $x_{i,2}\in P$ for some $i$; 
let $P_1$ be obtained from $P$ by replacing $x_{i,2}$ by $x_{i,1}$, then $P,P_1,Q$ is a locally connected path. 

Assume then $\p:=\ovl{P}\neq \ovl{Q}=:\q$, and let $\p,\p_1,\ldots,\p_r,\q$ be a locally connected path in $k[\Delta]$. If $\ssi{P}=c$ or $\ssi{Q}=c$, for every $h$ let $P_h$ be the prime with $\ovl{P_h}=\p_h$ and $\ssi{P_h}=c$. Then $P,P_1,\ldots,P_r,Q$ is a locally connected path. 
If $\ssi{P}=\ssi{Q}=c+1$, then $x_{i,2}\in P$ and $x_{j,2}\in Q$ for some $i,j$. If $i=j$, the same path connecting $\ovl{P}$ and $\ovl{Q}$ can be lifted to a path between $P$ and $Q$ where every $P_h$ in the path contains $x_{i,2}$ (and all other variables have second index 1). 
If $i\neq j$, since $\Delta$ is 2-locally connected there is a path  $\p,\p_1,\ldots,\p_r,\q$ such that each $\p_h$ contains at most one of the variables $x_i$ and $x_j$. As in the case $i=j$, this path can be lifted to a path between $P$ and $Q$, where $P_h$  contains $x_{i,2}$ or $x_{j,2}$ (and all other variables have second index 1) whenever $\p_h$ contains $x_i$ or $x_j$. This is a path because no $P_h$ contains both $x_{i,2}$ and $x_{j,2}$ so $\ssi{P_h}\leq c+1$ for every $h$, thus $P_h\in \Ass(k[\Delta^{(2)}])$ for every $h$.

\end{proof}

\section{Ordinary powers}

Our goal in this section is to prove Theorem \ref{cmordintro}, which we now recall. 

\begin{theorem}\label{main2}
	Let $\Delta$ be a pure simplicial complex. Then the following are equivalent:
	\begin{itemize}
		\item[$($a$)$] $\Delta$ is a complete intersection;
		\item[$($b$)$] $I_{\Delta}^{m}$ is Cohen--Macaulay for every $m\geq 1$;
		\item[$($c$)$] $R/I_{\Delta}^{m}$ satisfies $(S_2)$ for some $m\geq 3$. 
		\item[$($d$)$] $\Delta$ is a matroid and $I_{\Delta}$ has the K\"onig property.
	\end{itemize}	
\end{theorem}

\begin{Remark}\label{Konig}
Our addition, part (d), gives a new simple test to determine when $I_\Delta$ is a complete intersection: $\Delta$ needs to be a matroid, and the product of the variables has to be in $(I_\Delta)^c$, where $c$ is the co--rank of $\Delta$.
\end{Remark}

For the proof we will need a simple fact about star configurations. To streamline the proof, we will make use of the concept of simplified matroid and cover ideals. We start by discussing these ingredients.
 \medskip
 
The {\em monomial star configuration of grade $c$ in $R$} is the Stanley--Reisner ideal of the complete matroid of rank $n-c$, or equivalently, it is the ideal generated by all squarefree monomials of degree $n-c+1$. For instance, $\m=(x_1,\ldots,x_n)$ is the monomial star configuration of grade $n$, and $(x_1x_2\cdots x_n)$ is the one of grade 1. One can easily deduce a version of the following result (for any star configuration of hypersurfaces) using \cite[Thm~4.9]{Ma20}.  However, to keep the paper as self--contained as possible, we include here a short proof. 

\begin{lem}\label{stconf}
	Let $I$ be a monomial star configuration of grade $c$. Then $I^{(m)}= I^m$ for some $m\geq 2$ if and only if $c=1$ or $c=n$. 
\end{lem}

\begin{proof}
	First, for any ideal $J$ let $\alpha(J)$ be the smallest degree of a generator of $J$. In particular, if $I$ is an in the statement, then $\alpha(I^m)=m\alpha(I)=m(n-c+1)$. \\
``$\Lla$"	If $c=1$ or $c=n$, then $I=(x_1x_2\cdots x_n)$ or $I=\m$, so $I$ is a complete intersection, thus it is well--known that $I^m=I^{(m)}$ for all $m\geq 1$, see e.g. \cite{CN76}.
	
``$\Lra$" It suffices to show that if $2\leq c \leq n-1$, then $\alpha(I^{(m)})<\alpha(I^m)=m(n-c+1)$ for every $m\geq 2$. Take any $m\geq 2$, and write $m=qc+r$ for some $0\leq r \leq c-1$. If $r=0$, set $M:=(x_1x_2\cdots x_n)^q$, and if $r>0$, set $M:=(x_1x_2\cdots x_n)^q(x_1x_2\cdots x_{n-c+r})$. It is easily seen that $M\in \p^m$ for every monomial prime of grade $c$, i.e. $M\in \p^m$ for all $\p\in \Ass(R/I)$, i.e. $M\in I^{(m)}$. Now $\deg(M)=nq$ or $nq+(n-c+r)$, depending on whether $r=0$ or $r>0$. In either case $\alpha(I^{(m)})\leq \deg(M)<m(n-c+1)$.
	
If $r>0$ one deduces that $nq+(n-c+r)\geq m(n-c+1)=m(n-c)+m=(qc+r)(n-c)+qc+r$ if and only if $(n-c)(qc+r-q-1)\leq 0$. Since $n>c$, this is equivalent to $q(c-1)\leq 1-r$, and since $r\geq 1$ and $c-1\geq 1$ this only holds only if $q=0$ and $r=1$, that is, if $m=1$. The case $r=0$ is proved analogously.

\end{proof}

Next, we recall that the {\em cover ideal} $J(\Gamma)$ of a simplicial complex $\Gamma$ can be defined, for instance, as $I_{\Gamma^*}$, i.e. the Stanley--Reisner ideal of the naive dual of $\Gamma$. 

\begin{remark}\label{s-p}
	If $\Lambda$ is a loopless matroid, i.e. $i\in \Lambda$ for all $i\in [n]$, then its 1-skeleton is a complete $s$-partite graph, for some $s\geq c$. 
 The sets $A_1,\ldots,A_s$ forming the associated partition of $[n]$ are called the {\em parallel classes} of $\Gamma$. A matroid is {\em simplified} if its 1-skeleton is $K_n$, the complete graph on $[n]$, i.e. if there are $n$ parallel classes. A {\em simplified matroid} $^{\rm si}\Lambda$ associated to a matroid $\Lambda$ is obtained by choosing precisely one $a_i \in A_i$ for every $i=1,\ldots,s$, taking $W:=\{a_1,\ldots,a_s\}$ to be the vertex set of $^{\rm si}\Lambda$, and setting $(a_{i_1},\ldots,a_{i_t}) \in$  $^{\rm si}\Lambda$ if and only if $(a_{i_1},\ldots,a_{i_t}) \in \Lambda$. By construction, the 1-skeleton of $^{\rm si}\Lambda$ is $K_n$. 
	
If $\Lambda$ is a matroid then $^{\rm si}\Lambda$ is a matroid, and there is a faithfully flat map $\varphi:k[x_{a_1},\ldots,x_{a_t}] \lra k[x_1,\ldots,x_n]$ given by $f(x_{a_j})=\prod_{i \in A_j}x_i$ such that $\varphi(J(\Lambda))=J(^{\rm si}\Lambda)$ and $\varphi(J(\Lambda)^m)=J(^{\rm si}\Lambda)^m$ and $\varphi(J(\Lambda)^{(m)})=J(^{\rm si}\Lambda)^{(m)}$ for all $m\geq 1$ (see \cite[Rmk.~14 and Prop.~15]{CV15}, cf. \cite[Lem~3.1]{GH17}).
	
In particular,  $J(\Lambda)$ and $J(^{\rm si}\Lambda)$ have the same codimension and total Betti numbers, and one of them is Cohen--Macaulay, $(S_2)$ or complete intersection, resp. if and only if the other one is.	For instance, $J(\Lambda)$ is a complete intersection if and only if $J(^{\rm si}\Lambda)$ is a prime ideal.
	
Additionally, $I_{\Delta}$ is a complete intersection if and only if $\Delta$ is a matroid and has precisely $c$ parallel classes \cite[Rmk.~33]{CV15}.

\end{remark}

Recall that a squarefree monomial ideal $I=I_\Delta$ has the {\em K\"onig property} if it contains a monomial complete intersection of grade $\dim(\Delta^{\vee})+1$. This notion appears, for instance, in a long--standing conjecture raised by Conforti and Cornuejols in Combinatorial Optimization theory \cite[Conj.~1.6]{Co01}, which translates to the following algebraic conjecture: {\em If $I=I_\Delta$ is a squarefree monomial ideal, then $I^m=I^{(m)}$ for all $m\geq 1$ if and only if $I$ and all its minors have the K\"onig property}. (see e.g. \cite[Conj~4.7]{GRV}) Recall that a {\em minor} of $I$ is any ideal obtained from $I$ by setting a subset of the variables equal to 0 and another subset of the variables equal to 1. 

It is easily seen that, for any matroid $\Lambda$, the ideal $J(\Lambda)$ has the K\"onig property if and only $J(^{\rm si}\Lambda)$ does. We are now ready to prove Theorem \ref{main2}.

\begin{proof}[Proof of Theorem \ref{main2}] Without loss of generality we may assume $\Delta$ has no cone points; let $\dim(\Delta)=n-c-1$, and let $I:=I_{\Delta}$, so ${\rm grade}(I_\Delta)=c$. 
	It is well-known that (a)  $\Lra$ (b) \cite{CN76} and clearly (b) $\Lra$ (c). 
 
	(c) $\Lra$ (a). Since $R/I^m$ satisfies $(S_2)$, then $I^m$ is unmixed, so $I^{(m)}= I^m$. Since $R/I^{(m)}$ satisfies $(S_2)$, Theorem \ref{cmsymb} implies $\Delta$ is a matroid, and then so is its dual $\Delta^*$. Since $\Delta$ has no cone points, then $\Delta$ is a coloopless matroid $\Delta$, i.e. $\Delta^*$ is loopless. Then, by Remark \ref{s-p}, all the properties of (c) and (a) are preserved by passing from $\Delta^*$ to $^{\rm si}(\Delta^*)$, so we may assume the 1-skeleton of $\Delta^*$ is $K_n$. 
	
	We prove the statement by induction on $c\geq 2$. 	
	If $c=2$ then $\Delta^*=K_n$ so $I_{\Delta}$ is, in particular, a star configuration. Then since $I_{\Delta}^m$ is $(S_2)$ for some $m\geq 3$, we have $I^{m}=I^{(m)}$, so $I=\m$ by Remark \ref{stconf}.
	Assume $c\geq 3$, let $n$ be the smallest positive integer allowing a counterexample $\Delta$ on vertex set $[n$], i.e. the smallest $n\in \ZZ_+$ admitting a pure simplicial complex $\Delta$ on $[n]$ for which $I_{\Delta}^m$ is Cohen-Macaulay for some $m\geq 3$ and $I:=I_{\Delta}$ is not a complete intersection.

Let $P=(x_1,\ldots,x_{n-1})$, then in the regular local ring $R_P$ we have $I_P=(M_1/1,\ldots,M_r/1)$ where $M_1,\ldots,M_r$ are monomials in $S:=k[x_1,\ldots,x_{n-1}]$. It is well-known (e.g. it follows from the theory of *local ring as developed in \cite[Section~1.5]{BH93}) that if we set $\wdt{I}:=(M_1,\ldots,M_r)\subseteq S$, then $\wdt{I}$ is unmixed, Cohen--Macaulay, $(S_2)$, or a complete intersection if and only if $I_P$ is.

Since in our case, $I$ is an unmixed squarefree monomial ideal, then so is $I_P$ and thus, by the above, so is $\wdt{I}$. Note that $\wdt{I}=I:x_n$ or, equivalently, $\wdt{I}=I_{\Lambda}$ where $\Lambda:={\rm Del}_{\Delta}(\{n\})$ denotes the matroid on $[n-1]$ obtained by deleting $\{n\}$. Again by the above, since $I^m$ is Cohen-Macaulay, so is $\wdt{I^m}=(\wdt{I})^m$, which implies in particular that $(\wdt{I})^m=(\wdt{I})^{(m)}$ is Cohen--Macaulay. By minimality of $n$, the Cohen-Macaulay property of $(\wdt{I})^m$ implies that  $\wdt{I}$ is a complete intersection. Then, by Remark \ref{s-p}, $\Lambda$ has only $c$ parallel classes. Since the 1-skeleton of $\Lambda$ is obtained from the one of $\Delta$ by removing only the vertex $v_n$, and all edges passing through it, it follows that there are precisely $c+1$ parallel classes in $\Delta$. Therefore $n=c+1$. So if one sets $Q_i=(x_j\,\mid\, j\neq i)$ then it is clear that $\Ass(k[\Delta])\subseteq \{Q_1,\ldots,Q_{c+1}\}$. This inequality is strict for, if not, then $I=I_{\Delta}$ is a star configuration, thus Lemma \ref{stconf} would contradict the assumption. 
	
Without loss of generality we may  then assume $Q_{1}\notin \Ass(k[\Delta])$, and then $x_{1}\in \bigcap_{i=2}^{c+1} Q_i \subseteq I$. We can write $I_{\Delta}=(x_{1}, I_1)$ for an ideal $I_1$ of ${\rm grade}(J)=c-1$ extended from 
	$R_1:=k[x_2,\ldots,x_{c+1}]$. Since $x_1\notin R_1$ and $\Delta$ is a matroid, then $I_1$ is the Stanley--Reisner ring of a matroid in $c$ variables. Now  
	{\small	$$I^m=(I_1,x_1)^m =(I_1^m, I_1^{m-1}x_1,\ldots, I_1x_1^{m-1}, x_1^m) \quad \text{ and } \quad  I^{(m)} =(I_1^{(m)}, I_1^{(m-1)}x_1,\ldots, I_1x_1^{m-1}, x_1^m)$$}
	where the first equality follows from the definition of ordinary power, while the second one is proved, for instance, in \cite[Thm.~7.8]{BCG+} or \cite[Thm~3.4]{HNTT}. 
	So, if we endow $R=k[x_1,\ldots,x_{c+1}]$ with the grading $\deg(x_1)=1$ and $\deg(x_i)=0$ for every $i\geq 2$, then the two $R$-ideals  $I^m$ and $I^{(m)}$ are graded, and since they are equal, their components of degree 0 must be equal, so $I_1^m=I_1^{(m)}$. Since $I_1$ is the Stanley-Reisner ring of a matroid, it follows by Theorem \ref{cmsymb} that $I_1^{(m)}=I_1^m$ is Cohen--Macaulay and ${\rm grade}(I_1)=c-1$. By inductive hypothesis, $I_1$ is a complete intersection, and then so is $I=(x_1,I_1)$.
	
	(d) $\Lra$ (a). Since all properties of (d) and (a) are preserved by passing from $\Delta^*$ to $^{\rm si}(\Delta^*)$, then we may assume $\Delta^*$ is simplified. Let $M_1,\ldots,M_{c}\in I$ be a regular sequence of squarefree monomials. After possibly replacing $M_c$ by a squarefree monomial multiple of $M_c$, we may further assume $M_1M_2 \cdots M_{c} = x_1\cdots x_n$. 
	
	Since the product of the $M_j$'s is a squarefree monomial, then $\supp(M_1),\ldots,\supp(M_c)$ is a partition of $[n]$. Since each $M_i\in I$, then $\supp(M_i)\cap \p\neq \emptyset$ for every $\p\in \Ass(k[\Delta])$. Since $\p$ contains precisely $c$ variables, and $\supp(M_1),\ldots,\supp(M_c)$ are disjoint and they all meet $\p$, then $|\supp(M_i)\cap \p|=1$ for every $i$. Now, we prove that $|\supp(M_i)|=1$ for all $i$, which shows that $I=\m$, and in particular is a complete intersection. 
	Assume we have two distinct variables $x,y\in \supp(M_i)$. By Remark \ref{s-p}, since $\Delta^*$ is simplified then its 1-skeleton is $K_n$, and thus there exists a prime $\p\in \Ass(k[\Delta])$ containing $x$ and $y$, so $|\supp(M_i)\cap \p|\geq 2$, contradicting the above.

	(b) $\Lra$ (d). Surely $x_1x_2\cdots x_n\in \p^c$ for every monomial prime $\p$ of grade $c$; in particular, $x_1\cdots x_n \in I^{(c)}=I^c$, i.e. there exist monomials $M_1,\ldots,M_c\in I$ such that $M_1\cdots M_c = x_1\cdots x_n$. Since $x_1x_2\cdots x_n$ is squarefree, then $M_1,\ldots,M_c$ form a regular sequence of length $c={\rm ht}(I)$ in $I$, so $I$ has the K\"onig property. 
\end{proof}

\section{Regularity of symbolic powers}

In this section we provide an elementary proof for Theorem \ref{regintro}, i.e.  the formula, first found by Minh and Trung in \cite{MT17}, calculating the regularity of the symbolic powers of Stanley--Reisner ideals of matroids.

Consistently with \cite{MT17}, we define ${\rm core}(\Delta)$ to be the simplicial complex obtained by removing from $\Delta$ the cone points of $\Delta$. When $\Delta$ is a matroid, then ${\rm core}(\Delta)$ is the coloopless matroid obtained by removing all coloops of $\Delta$. 
\begin{theorem}\label{reg}
	Let $\Delta$ be a matroid, then $\reg(I_{\Delta}^{(m)}) = (m-1)c(\Delta) + r({\rm core}(\Delta)) + 1$.
\end{theorem}

We first prove a basic fact about matroids. Recall that the {\em circumference} $c(\Delta)$ of a matroid $\Delta$ is the maximum size of a circuit of $\Delta$.
\begin{lem}\label{c(Delta)}
	Let $\Delta$ be a coloopless matroid of rank $r(\Delta)=n-c$. Then, for all $0\leq j \leq c-2$
	$$
	c(\Delta) \geq \frac{n-j}{c-j}.
	$$
\end{lem}

\begin{proof}
	
	First assume $j=0$. Since $\Delta$ is a matroid, then it is a Cohen--Macaulay simplicial complex, so $\pd(R/I_\Delta)={\rm grade}(I_\Delta)=c$; additionally, since $\Delta$ is coloopless, it is well-known (e.g. \cite[Section~7]{St77}) that the last graded shift in a minimal graded free resolution $\mathbb F_\bullet$ of $R/I_{\Delta}$ is $n$. Since $\mathbb F_\bullet$ can be obtained by properly trimming Taylor's resolution of $R/I_\Delta$, then the above implies the existence of minimal monomial generators $u_1,\ldots,u_c$ of $I_{\Delta}$ with ${\rm lcm}(u_1,\ldots,u_c)=x_1\cdots x_n$. Then $x_1\cdots x_n$ divides $u_1\cdots u_c$, so
	$$
	n\leq \deg(u_1)+\ldots +\deg(u_c)\leq c\cdot c(\Delta).
	$$
	We can now prove the statement by induction on $c={\rm grade}(I_\Delta)\geq 2$. If $c=2$, then $j=0$, so the statement holds by the above. 
	
	If $j\geq 1$, let $u_1,\ldots,u_{c-1}\in G(I_{\Delta})$. If ${\rm lcm}(u_1,\ldots,u_{c-1})=x_1\cdots x_n$, then, as above, we are done. Without loss of generality, we may then assume $x_n\notin {\rm supp}(u_i)$ for any $1\leq i \leq c-1$. Let $\Delta +n$ be obtained by adding $n$ to every facet of $\Delta$
	not containing $n$. Since $\Delta +n$ is the dual of the contraction of $\Delta^*$ at $n$, which is a matroid, then also $\Delta+n$ is a matroid, having $n$ as the only cone point. Let $\Gamma$ be the contraction of $\Delta + n $ at $n$. Then $\Gamma$  is a coloopless matroid on $[n-1]$ with $c(\Delta)\geq c(\Gamma)$ and $\dim(\Gamma)=\dim(\Delta)=n-c-1$ . 
	In particular,  ${\rm grade}(I_{\Gamma}))={\rm grade}(I_\Delta)-1 = c-1$. Thus
	$$c(\Delta)\geq c(\Gamma) \geq \frac{(n-1-(j-1))}{c-1-(j-1)}=\frac{n-j}{c-j},$$ 
	where the rightmost inequality follows by induction.

\end{proof}

Next, we need the following crucial lemma.  Recall that the ideals $J_m\subseteq T$ and $C_M$ in the statement are defined in Notation \ref{J_m}. 
\begin{lem}\label{C_M}
	Let $\Delta$ be a coloopless matroid, i.e. $\Delta$ has no cone points, and let $m\in \ZZ_+$. For any circuit $U$ of $\Delta$, there exists $M\in G((J_m)^\vee)$ and an order of the variables of $R=k[x_1,\ldots,x_n]\supseteq I_\Delta$ such that, using the orders defined Notation \ref{J_m}, one has  
	$$
	{\rm grade}(C_M)=(m-1)|U| + (n-c).
	$$
\end{lem}

While the proof of this crucial equality is elementary, it relies on two delicate ingredients. One of them is a very careful choice on the order of the variables of $R$, which heavily depends on $U$. In fact, the order we illustrate in the proof seems to be the only one for which ${\rm grade}(C_M)$ is precisely $(m-1)|U|+(n-c)$. The second one is an explicit detailed description of all variables in the colon ideal $C_M$. Under our specific order of the variables, we show that the variables in $C_M$ with second index larger than 1 are precisely the ones whose first index is in the complement of the circuit $U$.

\begin{proof}
	Let $H:=U^*=[n]-U$. Since $U$ is a circuit, then $H$ is a hyperplane of $\Delta^*$.  (e.g. \cite[Prop.~2.16]{Ox92}.)
	We relabel the vertices of $[n]$ -- and, thus, the variables of $R$ -- as follows.
	Fix one independent set $I$ of $\Delta^*$ inside $H$ of maximal rank, i.e. of rank $c-1$, and label its elements as  $n-c+1,\ldots,n-1$. Then, fix any basis $B$ containing $I$. Notice that $B$ is obtained by adding one element to $I$; label such element as $n$. Finally choose any order on the remaining $n-c$ elements and label them as $1,\ldots,n-c$.
	
	Let $\p_B:=(x_{n-c+1},\ldots,x_n)$. Let $M:=x_{n-c+1,1}x_{n-c+2,1}\cdots x_{n-1,1}x_{n,m}$. By the above, $\p_B\in \Ass(k[\Delta])$, so $\p_M:=(x_{n-c+1,1},x_{n-c+2,1},\ldots, x_{n-1,1},x_{n,m})\in \Ass(k[\Delta^{(m)}])$ and then $M\in G((J_m)^\vee)$, as needed. \medskip

	By the proof of Theorem \ref{cmsymb} $(1) \Lra (2)$, the colon ideal $C_M$ is generated by a subset of the variables of $T$, which we now describe. 
	We first determine all variables with second index 1 in $C_M$. Clearly, $x_{n,1}\in C_M$, because we can replace $x_{n,m}$ in $M$ with $x_{n,1}$, i.e. we can write $x_{n,1}M=x_{n,m}M'$ and the monomial $M':=x_{n,1}M/x_{n,m}=x_{n-c+1,1}x_{n-c+2,1}\cdots x_{n-1,1}x_{n,1}$ is clearly in $G((J_m)^{\vee})$ and $M'<M$. Next, since $\Delta^*$ has no coloops, then for every variable $x_j$ with $1\leq j \leq n-c-1$, there is a prime ideal $\p_j\in \Ass(k[\Delta])$  containing $x_j$. Since $\Delta^*$ is a matroid, then $\p_B- (x_{n-h}) + (x_j) \in \Ass(k[\Delta])$ for some $0\leq h \leq c-1$. 
	
	If $h=0$, then the monomial $N:=x_{j,1}x_{i_1,1}x_{i_2,1}\cdots x_{n-1,1}=x_{n,m}M'$ where $M':=(x_{j,1}M)/x_{n,m}$, and if $1\leq h \leq c-1$, then $N:=x_{j,1}M=x_{n-h,1}M'$ where $M':=(x_{j,1}M)/x_{n-j,1}$. In either case,  $M'\in G((J_m)^{\vee})$ and $M'<M$, so $x_{j,1}\in C_M$ for all $1\leq j \leq n-c$. It follows that among all variables whose second index is 1, the ones in $C_M$ are precisely the following $n-c+1$ variables: $$\{x_{j,1}\,\mid\,1\leq j \leq n-c \text{ or }j=n\}.$$
	
	Next, consider the variables $x_{j,b}$ with second index $b$ for some $2\leq b \leq m$. Since the second index of $x_{j,b}$ is $>1$, and since $\ssi{M}=\ssi{\p_M}=m+c-1$, the variable $x_{j,b}$ can only replace $x_{n,m}$, i.e. it suffices to identify all monomials in $G((J_m)^{\vee})$ of the form $x_{n-c+1,1}x_{n-c+2,1}\cdots x_{n-1,1}x_{j,b}$ for some $j$. By Alexander duality, they correspond to the prime ideals of the form $(x_{n-c+1,1},\ldots,x_{n-1,1},x_{j,b})$ in $\Ass(k[\Delta^{(m)}])$. So, by Proposition \ref{polprimes}, we need to find all primes in $\Ass(k[\Delta])$ of the form $(x_{n-c+1},\ldots,x_{n-1},x_j)$. Since, by construction, $\{n-c+1,\ldots,n-1\}$ is an independent set in $\Delta^*$ of rank $c-1$, then $(x_{n-c+1},\ldots,x_{n-1},x_j)\in \Ass(k[\Delta])$ if and only if the rank of $\{n-c+1,\ldots,n-1,j\}$ in $\Delta^*$ is $c$, which happens if and only if $j\notin H$, so if and only if $j\in U$.  
	
	Therefore, for all $2\leq b\leq m$ we have $(x_{n-c+1,1},\ldots,x_{n-1,1},x_{j,b})\in \Ass(k[\Delta^{(m)}])$ if and only if $j\in U$. We only need to be aware that if $j=n$ and $b=m$, then this ideal is just $\p_M$, so $x_{n,m}\notin C_M$. Consequently,
	$$
	C_M = (x_{j,1}\,\mid\,j\notin\{n-c+1,\ldots,n-1\}) + (x_{j,b}\,\mid\,j\in U \text{ and }2\leq b \leq m) - (x_{n,m}),
	$$
	and, therefore, ${\rm grade}(C_M) = (n-c + 1) + |U|(m-1) -1 = (m-1)|U| + (n-c)$. 
\end{proof}

\begin{proof}[Proof of Theorem \ref{reg}]
	Since coloops are variables not appearing in $I_{\Delta}$, it suffices to prove the statement when $\Delta$ is a coloopless matroid on $[n]$. In this case, the rank of the core is $r({\rm core}(\Delta))=r(\Delta)$.
	So, if we set $c:={\rm grade}(I_{\Delta})$, and let $\omega(I_{\Delta})$ denote the largest degree of a minimal generator of $I_{\Delta}$, then we need to show
	$$
	\reg(I_{\Delta}^{(m)}) = (m-1)\omega(I_\Delta) + n-c + 1.
	$$
	To compute the left-hand side, we dualize the polarization and obtain $$\reg(I_{\Delta}^{(m)}) =\reg([I_{\Delta}^{(m)}]^{\pol})=\pd(T/([I_{\Delta}^{(m)}]^{\pol})^\vee)=\pd(([I_{\Delta}^{(m)}]^{\pol})^\vee)+1,$$
	where the last equality holds by a theorem of Terai \cite{Te99}. Recall that  
	\[G(J_m^\vee)=(x_{i_1,a_1}x_{i_2,a_2}\cdots x_{i_c,a_c}\,\mid\, (x_{i_1},\ldots,x_{i_c})\in \Ass(R/I_\Delta) \text{ and }\sum_{j=1}^ca_j\leq m+c).\]
	
	By Theorem \ref{cmsymb}.(2), the Alexander dual $J_m^\vee$ of $J_m=[I_{\Delta}^{(m)}]^{\pol}$ has linear quotients under the order discussed in Notation \ref{J_m}. Since the projective dimension of an ideal $L$ with linear quotients is the largest grade among the ideals $(N\in G(L)\,\mid\,N<M):M$ for $M \in G(L)$, then we need to prove that 
	\begin{equation}\label{goal}
		\max\{{\rm grade}(C_M)\,\mid\,M\in G((J_m)^\vee)\} = (m-1)\omega(I_\Delta) + n-c.
	\end{equation}
	
	``$\geq$": Let $U$ be a maximal size circuit of $\Delta$, so $|U|=c(\Delta)$. By Lemma \ref{C_M}, there exists $M\in G((J_m)^\vee)$ such that $\grade(C_M)=(m-1)|U|+n-c = (m-1)\omega(I_\Delta) + n-c$.\medskip
	
	``$\leq$: Let $M\in G((C_M)^{\vee})$ be such that ${\rm grade}(C_M)$ is maximal. We first prove the desired inequality when the second indices of $M$ are precisely 1,1,\ldots,1, $m$, i.e. when $M$ has the form $M=x_{i_1,1}x_{i_2,1}\cdots x_{i_{c-1},1}x_{i_{c},m}$ for some $1\leq i_1<i_2<\ldots<i_c\leq n$. By definition of $M\in G((J_m)^\vee)$, $\{i_1,\ldots,i_c\}$ is a basis $B$ of $\Delta^*$, so $I:=\{i_1,\ldots,i_{c-1}\}$ is an independent set in $\Delta^*$. Let $H$ be any hyperplane of smallest size containing $I$, and let $U:=[n]-H$ be its complement, so $U$ is a circuit in $\Delta$. Now, an argument similar to the one of Lemma \ref{C_M} shows that 
	{\small$$
		\begin{array}{ll}
			C_M & \subseteq (x_{j,1}\,\mid\,j\notin\{n-c+1,\ldots,n-1\}) + (x_{j,b}\,\mid\,j\in U \text{ and }2\leq b \leq m-1) + (x_{j,m}\,\mid\, j\in U \text{ and }j<i_c)\\
			& \subseteq (x_{j,1}\,\mid\,j\notin\{n-c+1,\ldots,n-1\}) + (x_{j,b}\,\mid\,j\in U \text{ and }2\leq b \leq m) - (x_{n,m})
		\end{array}$$}
	so, as in the proof of Lemma \ref{C_M}, one finds $\grade(C_M)\leq (m-1)|U| + n-c\leq (m-1)\omega(I_{\Delta}) + n-c$.\medskip
	
	To conclude, we now prove the inequality ``$\leq$" of (\ref{goal}) by induction on $m\geq 1$. The base cases are the monomials $M$ of the form $M=x_{i_1,1}x_{i_2,1}\cdots x_{i_{c-1},1}x_{i_{c},m}$. 
	
	Assume $m\geq 3$, we s and let $M=x_{i_1,a_1}x_{i_2,a_2}\cdots x_{i_{c},a_c}$ for some $1\leq a_1\leq a_2\leq \ldots \leq a_c$ with $a_1+\ldots +a_c=c+m-1$, and $c$ distinct integers $i_1,\ldots,i_c$ in $[n]$. 
	
	We define a set $D\subseteq [c]$ and $d\in \NN_0$ as follows: if $a_1\geq 2$ set $D=\emptyset$ and $d=0$; if $a_1=1$, set $D=\{h\in[c]\,\mid\,a_h=1 \}$ and $d=|D|\leq c$. If $d\geq c-1$, then $M$ is as in the base case, so the statement follows. Assume $d\leq c-2$. It is evident that $C_M\subseteq (x_{j,1}\,\mid\,j\in [n]-D) + (x_{j,b}\in C_M\,\mid\,b\geq 2)$, so $\grade(C_M)\leq (n-d) + \grade(x_{j,b}\in C_M\,\mid\,b\geq 2)$.
	
	We now study the second summand. Let $R'=k[x_j\,\mid\,j\notin D]$, $m':=m-(c-d)$ and $\Lambda$ be the matroid obtained by deleting $\{i_1,\ldots,i_d\}$ from $\Delta$. Observe that $\Lambda^*$ is the contraction $\Lambda^*:=\Delta^* / \{i_1,\ldots,i_d\}$, so $\{i_{d+1},\ldots,i_c\}$ is a basis of $\Lambda^*$, and thus $M':=x_{i_{d+1},(a_{d+1}-1)}x_{i_2,(a_{d+2}-1)}\cdots x_{i_{c},(a_c-1)}$ is a minimal generator of $[(I_\Lambda)^{(m')}]^{\vee}$. As above, we write $C_{M'} := (N'\in G([(I_\Lambda)^{(m')}]^{\vee})\,\mid\,N'<M')\;:_{R'} \; M'$. 
	Since the sum of the second indices in any element in $G((J_m)^\vee)$ is maximal,  then for all $1\leq h\leq d$ and $\ell\geq 2$ the variable $x_{i_h,\ell}$ is not in $C_M$.
	Therefore, by shifting the second index by one, it is immediately seen that the second summand can be re-written as
	$$(x_{j,b}\in C_M\,\mid\,b\geq 2) = (x_{j,b'+1}\,\mid\,x_{j,b'}\in C_{M'})$$
	therefore,
	$$
	\begin{array}{ll}
		\grade(C_M) &  \leq (n-d)+ \grade(C_{M'})\\
		& \leq (n-d)+ (m'-1)\omega(I_{\Lambda}) + (n-c)\\
		& \leq (n-d)+ (m-c+d-1)\omega(I_{\Gamma}) + (n-c)\\
		& \leq (m-1)\omega(I_{\Delta}) + (n-c),
	\end{array}
	$$
	where the second inequality holds by induction, the third one follows because $\omega(I_{\Lambda})\leq \omega(I_{\Gamma})$ (e.g. \cite[3.1.14]{Ox92}), and the last inequality follows from the inequality $(c-d)\omega(I_{\Delta})\geq n-d$, which, by Lemma \ref{c(Delta)}, holds for all $d\leq c-2$.

\end{proof}

\section{The Level Property}

In this section we prove Theorem \ref{levelintro}, which we recall here for the reader's convenience. The concept of level algebra was introduced by R. Stanley, see \cite[Section~3]{St77}. Let $J\subseteq R$ be a homogeneous Cohen--Macaulay ideal of grade $c$. The algebra $R/J$ is {\em level} if its canonical module $\omega_{R/J}\cong {\rm Ext}_R^c(R/J,R)$ is generated in single degree or, equivalently, if the graded $R$-module ${\rm Tor}^R_c(R/J,k)$ is generated in a single degree.

\begin{theorem}\label{levelthm}
Let $\Delta$ be a simplicial complex. The following are equivalent:
\begin{enumerate}
\item[$(1)$] $\Delta$ is a complete intersection and its minimal nonfaces all have the same size.
\item[$(2)$] $R/I_{\Delta}^m$ is a level algebra for all $m \ge 1$.
\item[$(3)$] $R/I_{\Delta}^m$ is a level algebra for some $m \ge 3.$
\item[$(4)$] $R/I_{\Delta}^{(m)}$ is a level algebra for all $m \ge 1$.
\item[$(5)$] $R/I_{\Delta}^{(m)}$ is a level algebra for some $m \ge 3$.
\end{enumerate}

\end{theorem}

It is natural to ask whether a similar statement is true if $R/I_{\Delta}^{(2)}$ is a level algebra. In this case, it is known that $\Delta$ need not be a complete intersection, see for instance  \cite[Section~5]{MTT19}, which is dedicated to study of 1-dimensional matroidal simplicial complexes for which  $R/I_{\Delta}^{(2)}$ is level. To complete the picture, we propose the following conjecture, whose forward implication follows immediately from Proposition \ref{equigen} below. 

\begin{conjecture}\label{level2}
	Let $\Delta$ be a matroid. Then $R/I_{\Delta}^{(2)}$ is a level algebra if and only if all circuits of $\Delta$ have the same size.
\end{conjecture}

As mentioned above, we offer a proof of Theorem \ref{levelthm} in a slightly different spirit from the previous sections, with the goal of illustrating how our approach can be combined with other types of techniques, e.g. Hochster's formula and the combinatorics of $\Delta^{(m)}$. We begin with a few preparatory results, starting with the following elementary observation: 
\begin{lem}\label{noboundary}

Let $(C_{\bullet},\partial)$ be the chain complex of the simplicial complex $\Delta$, and consider an element $\epsilon:=\sum_{|F|=t} c_F F \in C_t$. If $\epsilon$ is a boundary in $C_{\bullet}$, then $c_G=0$ for any facet $G \in \Delta$ with $|G|=t$.

\end{lem}

\begin{proof}

As $\epsilon$ is a boundary, there is a $b=\sum_{|\sigma|=t+1} r_{\sigma} \sigma$ so that $\partial(b)=\epsilon$. But $\partial(b)$ is supported only on faces which are contained in faces of size $|t+1|$. In particular, $G$ is not in the support of $\partial(b)$. As $\{\tau \in \Delta \mid |\tau|=t\}$ forms a basis for $C_t$, it follows that $c_G=0$.

\end{proof}

The following is well-known: 
\begin{lem}\label{chain}

Let $\Delta$ and $\Gamma$ be simplicial complexes. Then there is a natural isomorphism of chain complexes
\[C^{\Delta}_{\bullet} \otimes_k C^{\Gamma}_{\bullet} \to C^{\Delta*\Gamma}_{\bullet}\] given on basis vectors by $F \otimes G \mapsto F \cup G$.

\end{lem}

Recall that a matroid $\Delta$ is coloopless if, when regarded as a simplicial complex, it has no cone points. A circuit of $\Delta$ is a minimal non-face of $\Delta$.
\begin{prop}\label{circuit1}

Suppose $\Delta$ is a coloopless matroid and $C$ a circuit of $\Delta$. Let $\alpha_C:=\{(i,j)\, \mid\, i \in C,\, 1 \le j \le m\}$ and let $W:=\{(i,1) \mid i \notin C\} \,\cup\, \alpha_C \subseteq \Delta^{(m)}$. Then $\tilde{H}_{(m-1)|C|+d-1}(\Delta^{(m)}|_{W}) \ne 0$. 
    
\end{prop}

\begin{proof}
For any $F \in \Delta$, we set $F^{(1)}:=\{(i,1) \mid i \in F\}$.\\
Let $\Gamma_C:=\Skel^{(m|C|-2)}(\langle \alpha_C \rangle)$, i.e. $\Gamma_C$ is the standard codimension $1$ sphere on $\alpha_C$, and let $\beta_C:=\bigcup_{j \in C} \lk_{\Delta}(C-\{j\})^{(1)}$. We claim that $\beta_C*\Gamma_C$ is a subcomplex of $\Delta^{(m)}|_{W}$. To see this, take a facet $G:=F^{(1)} \cup (\alpha_C-\{(i,j)\}) \in \beta_C*\Gamma_C$. As $F \in \mathcal{F}(\lk_{\Delta}(C-\{k\}))$ for some $k \in C$, then  $\sigma:=F \cup (C-\{k\})$ is a basis of $\Delta$. Thus, there is $\p_{\sigma}\in \Ass(k[\Delta])$ for which $\{x_l\,\mid\,l\in C\}\cap \p_{\sigma}=\{x_k\}$. As $C$ is a circuit of $\Delta$, then $C-\{i\}$ is an independent set of $\Delta$, so we may find a basis $\tau$ of $\Delta$ containing it. Then $\{x_l\,\mid\,l\in C\}\cap \p_{\tau}=\{x_i\}$. As $\Delta$ is a matroid, there is an $x_l \in \p_{\sigma}$ so that $\p_{\sigma}-(x_l)+(x_i)$ and $\p_{\tau}-(x_i)+(x_l)$ are associated primes of $k[\Delta]$. But if $l \ne k$, then $\p_{\tau}-(x_i)+(x_l)$ would contain no element of $C$, which would force $C$ to be a face of $\Delta$. Thus we must have $l=k$, so in particular $\p_{\sigma}-(x_k)+(x_i) \in \Ass(k[\Delta])$. Then $(x_{l,1} \mid x_l \in \p_{\sigma}-(x_k))+(x_{i,m}) \in \Ass(k[\Delta^{(m)}])$. Since the facet in $\Delta^{(m)}$ corresponding  to this prime contains $G$ by construction, then $G \in \Delta^{(m)}|_W$, establishing that $\beta_C*\Gamma_C$ is a subcomplex of $\Delta^{(m)}|_W$. 

Next we claim that $G:=F^{(1)} \cup \alpha_C-\{(i,m)\}$ is a facet of $\Delta^{(m)}|_W$ for any facet $F \in \lk_{\Delta}(C-\{i\})$. Note that the complement of $G$ as a face of $\Delta^{(m)}|_W$ is $\{(j,1) \mid j \notin F\} \cup \{(i,m)\}$. If there is a larger face of $\Delta|_W$ containing $G$, it must thus contain $G$ and either $(i,m)$ or some $(j,1)$ with $j \notin F \cup C$. But it cannot contain $(i,m)$ since every associated prime of $\Delta^{(m)}$ must contain some element of $\alpha_C$. As $G$ is a face of $\Delta^{(m)}$, there is a facet $\sigma$ of $\Delta^{(m)}$ containing it. In particular, $x_{i,m} \in \p_{\sigma}$, and it follows $x_{a,b} \in \sigma$ whenever $a \ne i$ and $b \ge 2$. If $x_{a,1} \in \sigma$ for some $a \notin F$, then it follows $\p_{G-\{(i,1)\}+\{(j,m)\}} \in \Ass(k[\Delta^{(m)}])$, so $\p_{F \cup C-\{i\} \cup \{j\}} \in \Ass(k[\Delta])$. But then $F \cup C-\{i\} \cup \{j\}$ is a basis of $\Delta$, which in particular implies $F \cup \{j\} \in \lk_{\Delta}(C-\{i\})$, contradicting that $F$ is a facet of $\lk_{\Delta}(C-\{i\})$. Therefore, $G$ is a facet of $\Delta^{(m)}|_W$.   

Now, since $\Delta$ is a coloopless matroid, then $\tilde{H}_{d-|C|}(\lk_{\Delta}(C-\{j\}) \ne 0$. Let $w=\sum_{F \in \mathcal{F}(\lk_{\Delta}(C-\{j\})} c_F F$ be a cycle corresponding to a nonzero element of this homology with a nonzero term $c_F F$. Then $w$ is also a cycle of $\bigcup_{i \in C} \lk_{\Delta}(C-\{i\})$. As $\Gamma_C$ is the standard codimension $1$ sphere, it has the canonical generator $z$ for $\tilde{H}_{m|C|-2}(\Gamma_C)$ that is the alternating sum across all its facets. In particular, $z$ is fully supported. Then $z \otimes w$ is nonzero and by Lemma \ref{chain} corresponds to a cycle $\bar{z}$ in $C^{\Gamma_C*\beta_C}$, which is also a cycle in $C^{\Delta^{(m)}|_W}$ that is supported on the facet $F^{(1)} \cup \alpha_C-\{(j,m)\}$. By Lemma \ref{noboundary}, $\bar{z}$ cannot be a boundary, and it follows $\tilde{H}_{(m-1)|C|+d-1}(\Delta^{(m)}|_W) \ne 0$, as desired. 
\end{proof}

We next prove one implication of Conjecture \ref{level2}.
\begin{prop}\label{equigen}
Suppose $\Delta$ is a matroid. If $R/I_{\Delta}^{(m)}$ is level for some $m \ge 2$, then all circuits of $\Delta$ have the same size. 
  
\end{prop}

\begin{proof}

Let $C$, $C'$ be any two circuits of $\Delta$. By Proposition \ref{circuit1}, we have $\tilde{H}_{(m-1)|C|+d-1}(\Delta^{(m)}|_W)$ and $\tilde{H}_{(m-1)|C'|+d-1}(\Delta^{(m)}|_{W'})$ are nonzero for some $W, W'$ with $|W|=(m-1)|C|+n$ and $|W'|=(m-1)|C'|+n$. Since $R/I_{\Delta}^{(m)}$ is level, so is $T/[I_{\Delta}^{(m)}]^{\pol}$, by Proposition \ref{polcm} (4). By Hochster's formula \cite[Theorem 5.5.1]{BH93}, it follows that $|W|=|W'|$. Since $m \ge 2$, it follows that $|C|=|C'|$, so that $I_{\Delta}$ is generated in a single degree.

\end{proof}

We will refer frequently to the following well--known, elementary lemma.
\begin{lem}\label{primehastwo}
Let $\Delta$ be a matroid and let $C\neq C'$ be two circuits. Then every basis of $\Delta^*$ contains at least two distinct elements of $C \cup C'$.     
\end{lem}

\begin{proof}

Let $F$ be a basis of $\Delta^*$. Assume by contradiction $F\cap (C\cup C')=\{a\}$. Since $C$ and $C'$ are circuits of $\Delta$, then $F$ contains at least one element from each, so we must have $a \in C \cap C'$. Since $\Delta$ is a matroid, then there is a circuit $C''\subseteq (C\cup C')-\{a\}$. Since $F$ and $C''$ are not disjoint, there is $b\in F\cap C''$. So $b\in F\cap \left((C\cup C')-\{a\}\right)=\emptyset$, a contradiction. 
\end{proof}

\begin{lem}\label{maximaloverlap}

Let $\Delta$ be a matroid whose circuits all have the same size. Suppose $C\neq C'$ are circuits for which $|C \cap C'|$ is maximal. Then $(C \cup C')-\{i,j\}$ is an independent set for any $i \in C$ and $j \in C'-C$.  

\end{lem}

\begin{proof}

If $(C \cup C')-\{i,j\}$ is a dependent set of $\Delta$ then there is a circuit $C'' \subseteq (C \cup C')-\{i,j\}$. Since $j \notin C$, then $C \cup C'' \subseteq (C \cup C')-\{j\}$. In particular, $|C \cup C''|<|C \cup C'|$, thus
$$
|C\cap C'|=|C|+|C'| - |C\cup C'|<|C|+|C'| - |C\cup C''| = |C|+|C''| - |C\cup C''|=|C\cap C''|,
$$
where the second equality from the right follows because, by assumption, $|C'|=|C''|$. The above contradicts the maximality of $|C\cap C'|$, therefore $(C \cup C')-\{i,j\}$ is an independent set of $\Delta$.

\end{proof}

\begin{lem}\label{doubleswap}
Let $\Delta$ be a matroid and let $C\neq C'$ be two circuits. Suppose there are bases of the form $F \cup (C \cup C')-\{x,y\}$ and $G \cup (C \cup C')-\{a,b\}$, for some independent sets $F,G\in \Delta$, and some vertices $x,a \in C$ and $y,b \in C'-C$.
 Then $F \cup (C \cup C')-\{a,b\} \in \Delta$.     
\end{lem}

\begin{proof}
Let $A:=(F \cup (C \cup C')-\{x,y\})^*$ and $B:=(G \cup (C \cup C')-\{a,b\})^*$ be the complements of the given bases, so $A,B$ are bases of $\Delta^*$. If $a \in A$, then we may suppose $b \notin A$, else there is nothing to prove. Since $x$ is the unique element of $C$ contained in $A$, we must have $x=a$. As $\Delta^*$ is a matroid, and as $b \in B-A$, there is an $i \in A-B$ so that $A-\{i\} \cup \{b\}$ and $B-\{b\} \cup \{i\}$ are facets of $\Delta^*$. If $i \ne y$, then $B-\{b\} \cup \{i\}$ contains only one element of $C \cup C'$, violating Lemma \ref{primehastwo}. Then $i=y$, and we have the claim in the case $a \in A$. 
Now suppose $a \notin A$. Again appealing to the matroid property of $\Delta^*$, there an a $j \in A-B$ so that $A-\{j\} \cup \{a\}$ and $B-\{a\} \cup \{b\}$ are facets of $\Delta^*$. If $j \ne x$, then as $x$ is the unique element of $C$ contained in $A$, and as $a$ is the unique element of $B$ contained in $C$, we would have $B-\{a\} \cup \{j\}$ contains no element of $C$. But this cannot be, since $C$ is a minimal nonface of $\Delta$. Thus $j=x$, and replacing $A$ by $A-\{x\} \cup \{a\}$, we have reduced to the case where $a \in A$, completing the proof. 
\end{proof}

In what follows, for $C \subseteq [n]$ and $i \ge 1$ we set $\alpha_C^i:=\{(a,b) \mid a \in C, 1 \le b \le i\}$. When $F \in \Delta$, we set $F^{(1)}:=\{(i,1) \mid i \in F\}$. 

\begin{prop}\label{ci1}

Let $\Delta$ be a matroid. Suppose $C\neq C'$ are overlapping circuits for which $|C \cap C'|$ is maximal. For any $m \ge 1$, let $W:=\alpha_C^{m-1} \cup \alpha^2_{C'-C} \cup \{(i,1) \mid i \notin C \cup C'\}$. If $m \ge 3$, then 
\[\tilde{H}_{(m-1)|C|-|C \cap C'|+d-1}(\Delta^{(m)}|_W) \ne 0.\]
    
\end{prop}

\begin{proof}

Let 
\[\Gamma:=\Skel^{((m-1)|C|-2)}(\langle \alpha_C^{m-1} \rangle)*\Skel^{(2(|C'-C|)-2)}(\langle \alpha^2_{C'-C} \rangle)*(\bigcup_{\substack{U \in C \cup C' \\ |U|=2}} \lk_{\Delta}((C \cup C')-U)\]

We first claim that $\Gamma$ is a subcomplex of $\Delta^{(m)}|_W$. To see this take 
\[G:=F^{(1)} \cup \alpha_C^{m-1} \cup \alpha^2_{C'-C}-\{(i,a)\}-\{(j,b)\}\] where $F \in \mathcal{F}(\lk_{\Delta}((C \cup C')-\{x,y\})$, where $i \in C$ and where $j \in C'-C$. By Lemma \ref{doubleswap}, we have $F \cup ((C \cup C')-\{i,j\}) \in\mathcal{F}(\Delta)$. Then $\p_F+(x_i,x_j) \in \Ass(k[\Delta])$ so $(x_{\ell,1} \mid \ell \notin F)+(x_{i,m-1})+(x_{j,2}) \in \Ass(k[\Delta^{(m)}])$. It follows from construction that $G$ is contained in the facet of $\Delta^{(m)}$ corresponding to this prime, so $G \in \Delta^{(m)}|_W$.

Now we claim, for any $i \in C$ and $j \in C'-C$ and any $F \in \mathcal{F}(\lk_{\Delta}((C \cup C')-\{i,j\}))$ that $G:=F^{(1)} \cup \alpha_C^{m-1} \cup \alpha^2_{C'-C}-\{(i,m-1)\}-\{(j,2)\}$ is a facet of $\Delta^{(m)}|_W$. To see this, we note that the complement of $G$ in $\Delta^{(m)}|_W$ is $(F^*)^{(1)} \cup \{(i,m-1)\} \cup \{(j,2)\}$. If $G$ is not a facet of $\Delta^{(m)}|_W$ then there must be some facet $\sigma$ of $\Delta^{(m)}$ containing both $G$ and some element of $(F^*)^{(1)} \cup \{(i,m-1)\} \cup \{(j,2)\}$. If $x_{a,1} \in \sigma$ for some $a \notin F$, then $\bar{\sigma}$ -- which, let us recall, is the image of $\sigma$ under the projection $\Delta^{(m)}\lra \Delta$ onto the first coordinate -- is a facet of $\Delta$ containing $F$, $a$, and $(C \cup C')-\{i,j\}$ which contradicts the assumption that $F \in \mathcal{F}(\lk_{\Delta}((C \cup C')-\{i,j\}))$. Otherwise, $\sigma$ must contain either $(i,m-1)$ or $(j,2)$. If $(i,m-1) \in \sigma$, then we note, by Lemma \ref{primehastwo}, that every facet of $\Delta$ must exclude at least two elements of $C \cup C'$, one of $C$ and one of $C'$. Noting that $j \in C'-C$, there must thus be some $b \in C$ with $(b,\ell) \notin \sigma$. But as $\sigma$ contains $G$, it must be that $\ell=m$. However, $(j,2)$ is also not in $\sigma$ which contradicts the definition of $\Delta^{(m)}$. On the other hand, if $(j,2) \in \sigma$, then from above we have $(i,m-1) \notin \sigma$, so by Lemma \ref{primehastwo} there must be some $(u,2) \notin \sigma$ with $u \in C \cup C'$. But as $m \ge 3$, any such element is already contained in $G$, and it follows that $G$ is a facet of $\Delta^{(m)}|_W$.

Now, $\Delta$ is a matroid which is not a cone, thus $\lk_{\Delta}((C \cup C')-\{i,j\})$ is not acyclic for any $i \in C$ and $j \in C'-C$, noting that $C \cup C'-\{i,j\} \in \Delta$ by Lemma \ref{maximaloverlap}. As $\Skel^{(m-1)|C|-2}(\alpha^{m-1}_C)$ and $\Skel^{(2|C'-C|-2)}(\alpha^2_{C'-C})$ are codimension $1$ spheres, they each have homology generators with full support. Since $\Gamma$ is a subcomplex of $\Delta^{(m)}|_W$ then, by Lemma \ref{chain},  $\Delta^{(m)}|_W$ has a cycle in degree $(m-1)|C|-|C \cap C'|+d-1$ with support on a facet of the form $F^{(1)} \cup \alpha_C^{m-1} \cup \alpha^2_{C'-C}-\{(i,m-1)\}-\{(j,2)\}$. In particular Lemma \ref{noboundary} implies this element cannot be a boundary, so $\tilde{H}_{(m-1)|C|-|C \cap C'|+d-1}(\Delta^{(m)}|_W) \ne 0$.

\end{proof}

\begin{proof}[Proof of Theorem \ref{levelthm}] 
The implications $(2) \Rightarrow (3)$ and $(4) \Rightarrow (5)$ are obvious, $(1) \Rightarrow (2)$ follows essentially from the results in \cite{BE75}, see e.g. \cite[Thm.~2.1]{Sr89} or \cite[Thm.~2.1]{GV05}. $(3) \Rightarrow (4)$ follows from Theorem \ref{main2}. Thus we need only concern ourselves with $(5) \Rightarrow (1)$. Suppose $R/I_{\Delta}^{(m)}$ is level. It follows from Theorem \ref{cmsymb} that $\Delta$ is a matroid, and from Proposition \ref{equigen} that the circuits of $\Delta$ have the same size, i.e. $I_\Delta$ is equigenerated.

We note from Lemma \ref{polcm} that $R/I_{\Delta}^{(m)}$ is level if and only if $\Delta^{(m)}$ is level. Now, if $\Delta$ is not a complete intersection, then there exists two distinct circuits $C\neq C'$ of $\Delta$ with nonempty overlap, and we may pick $C$ and $C'$ to maximize the size of the overlap.  By Proposition \ref{circuit1}, there is a $W$ with $|W|=(m-1)|C|+n$ for which $\tilde{H}_{|W|-c-1}(\Delta^{(m)}|_W) \ne 0$, and by Proposition \ref{ci1} there is a $W'$ with $|W'|=(m-1)|C|-|C \cap C'|+n$ for which $\tilde{H}_{|W'|-c-1}(\Delta^{(m)}|_{W'}) \ne 0$. Since $|C \cap C'|>0$, then we have $|W| \ne |W'|$. But then Hochster's formula \cite[Theorem 5.5.1]{BH93} implies that $ \Delta^{(m)}$ is not level. Then there cannot be overlapping circuits in $\Delta$, i.e. $\Delta$ is a complete intersection, completing the proof.

\end{proof}

\section*{Acknowledgements} 
This material is based upon work supported by the National Science Foundation under Grant No. DMS-1928930 and by the Alfred P. Sloan Foundation under grant G-2021-16778, while the first author was in residence at the Simons Laufer Mathematical Sciences Institute (formerly MSRI) in Berkeley, California, during the Spring 2024 semester. 

The second author was partly supported by Simons Foundation Grant \#962192.


\bibliographystyle{plain}

\begin{thebibliography}{AAAA}
	\bibitem{AV79} { R. Achilles and W. Vogel, \"Uber vollst\"andige Durchscnitte in lokalen Ringem, {\em Math. Nachr.} {\bf 89} (1979), 285--298.}	


 \bibitem{AFH24}{ A. Almousa, G. Fl{\o}ystad and H. Lohne, Polarizations of powers of graded maximal ideals, {\em J. Pure Appl. Algebra} {\bf  226} (2022), Paper No. 106924, 33 pp.}


\bibitem{Ap45}{ R. Ap\'ery: Sur les courbes de premi\`ere esp\`ece de l’espace \` trois dimensions, {\em C.R. Acad. Sci. Paris} {\bf 220} (1945), 271--272.}
 
	\bibitem{Bj92}{ A. Bj\"{o}rner, The homology and shellability of matroids and geometric
		lattices, {\em Matroid applications}, Encyclopedia Math. Appl. {\bf 40} (1992) , Cambridge Univ. Press, Cambridge, 226--283.
	}
	
	
\bibitem{BCG+}{C. Bocci, S. Cooper, E. Guardo, B. Harbourne, M. Janssen, U. Nagel, A. Seceleanu, A. Van Tuyl, and T. Vu, The Waldschmidt constant for squarefree monomial ideals, {\em J. Algebraic Combin.} {\bf 44} (2016), 875--904.}	
	
		\bibitem{Bo48} {K. Borsuk, \emph{On the imbedding of systems of compacta in simplicial              complexes}, Fund. Math. \textbf{35} (1948), 217--234.}
	
	
\bibitem{BH93}{ W. Bruns and J. Herzog, Cambridge Studies in Advanced Mathematics, {\bf 39}, Cambridge University Press, Cambridge, 1993.}	
	
 \bibitem{BE75}{ D. A. Buchsbaum and D. Eisenbud, , Generic free resolutions and a family of generically
 	perfect ideals, {\em Adv. Math.} {\bf 18} (1975), 245--301.}

 
 \bibitem{Co01}{ G. Cornu\'ejols, Combinatorial optimization, {\em CBMS-NSF Regional Conference Series in Applied Mathematics}, {\bf 74}, Society for Industrial and Applied Mathematics (SIAM), Philadelphia, PA, 2001, Packing and covering.}
 
 
	\bibitem{CV15}{ A. Costantinescu and M. Varbaro, {$h$-}vectors of matroid complexes, {\em Combinatorial methods in topology and algebra}, Springer INdAM series {\bf 12} (2015), 203--227.	}
	
	
	\bibitem{CN76}{	R.C. Cowsik and M.V. Nori, Fibers of blowing up, {\em J. Indian Math. Soc.} {\bf 40}
		(1976) 217--222.}
	
	\bibitem{Da60}{ E. C. Dade, Multiplicity and monoidal transformations, {\em Ph.D. Thesis -- Princeton University.} 1960. 93 pp.}
	
	\bibitem{DD20}{ H. Dao, J. Doolittle and J. Lyle, Minimal {C}ohen-{M}acaulay simplicial complexes, {\em SIAM J. Discrete Math.} {\bf 34} (2020), 1602--1608.}
	
	\bibitem{Fa06}{ S. Faridi, Monomial ideals via square-free monomial ideals, {\em Lect. Notes Pure Appl. Math.} {\bf 244} (2006), 85--114.}

\bibitem{FKRS}{ S. Faridi, P. Klein, J. Rajchgot and A. Seceleanu, Polarization and Gorenstein liaison, preprint available at \href{https://arxiv.org/pdf/2406.19985}{https://arxiv.org/pdf/2406.19985}.}


	\bibitem{Fr08} { C. Francisco,  Tetrahedral curves via graphs and Alexander duality, {\em J. Pure Appl. Algebra} {\bf 212} (2008), 364--375.}
 
	\bibitem{Fr88} { R. Fr\"{o}berg, \emph{A note on the {S}tanley-{R}eisner ring of a join and of a suspension}, Manuscripta Math. \textbf{60} (1988), no.~1, 89--91.}  \MR{920761}

 \bibitem{Ga52}{F. Gaeta,  D\'etermination de la chaine syzyg\'etique des id\'eaux matriciels parfaits et son application \`a la postulation de leurs vari\'et\'es alg\'ebriques associ\'ees, {\em C.R. Acad. Sci. Paris} {\bf 234} (1952), 1833--1835.}

        
        \bibitem{GH17}{ A. Geramita, B. Harbourne, J. Migliore and U. Nagel, Matroid configurations and symbolic powers of their ideals, {\em Tran. Amer. Math. Soc.} {\bf 369} (2017), 7049--7066}
        
      \bibitem{GRV}{ I. Gitler, E. Reyes and R. H. Villarreal, Blowup algebras of square-free monomial ideals and some links to combinatorial optimization problems, 
      	{\em Rocky Mountain J. Math.} {\bf 39} (2009), 71--102.}  
 
	\bibitem{Gr60}{ A. Grothendieck, \'El\'ements de g\'eom\'etrie alg\'ebrique, {\em Inst. Hautes Études Sci. Publ. Math.}, no. 4,, 8, 11, 17, 20, 24, 28, 32, IHES, 1960. }
	
		\bibitem{GV05}{ E. Guardo and A. Van Tuyl, Powers of complete intersections: graded Betti numbers and applications. {\em Illinois J. Math.} {\bf 49} (2005), 265--279.}
	
	
\bibitem{HNTT}{ H. T. Hà, H. D. Nguyen, N. V. Trung and N. Tran, 
	{\em Math. Z.} {\bf 294} (2020), 1499--1520.}	
	
	
		\bibitem{Ha62}{ R. Harthsorne, Complete Intersections and Connectedness,
		 {\em Amer. J. Math} {\bf 84} (1962), 497--508.}
	
	\bibitem{HO80}{ M. Herrmann and U. Orbanz, Faserdimensionen von Aufblasungen lokaler Ringe und \"Aquimultiplizitat, {\em J. Math. Kyoto Univ.} {\bf 20} (1980), 651--659.
}
	
	\bibitem{HT02}{ J. Herzog and Y. Takayama, Resolutions by mapping cones, {\em The Roos Festschrift. Homology, Homotopy Appl.} {\bf 4} (2002), 277--294.}

        \bibitem{HTT05}{ J. Herzog, Y. Takayama and N. Terai, On the radical of a monomial ideal, {\em Arch. Math.} {\bf 85} (2005), 397--408.}
		
	\bibitem{HH94}{ M. Hochster and C. Huneke, Indecomposable canonical modules and connectedness, {\em Commutative algebra: syzygies, multiplicities, and birational 	algebra ({S}outh {H}adley, {MA}, 1992).} Contemp. Math. {\bf 159} (1994), 197--208.}
	
	\bibitem{Ho18}{ B. Holmes, On the diameter of dual graphs of {S}tanley-{R}eisner rings
		and {H}irsch type bounds on abstractions of polytopes,  {\em Electron. J. Combin.} {\bf 25} (2018), paper 1.60.}	

	\bibitem{HL21}{ B. Holmes and J. Lyle, Rank Selection and Depth Conditions for Balanced Simplicial Complexes,  {\em Electron. J. Combin.} {\bf 28} (2021), paper 2.28.}	

\bibitem{HU87}{ C. Huneke and B. Ulrich, The structure of linkage, {\em Ann. Math } {\bf 126} (1987),  277--334.}

\bibitem{KMMN01}{ J. O. Kleppe, J. C. Migliore, R. Mir\'o-Roig, U. Nagel, and C. Peterson, Gorenstein
	liaison, complete intersection liaison invariants and unobstructedness, {\em Mem. Amer. Math. Soc.} 	{\bf 154} (2001), no. 732, viii+116.}

        \bibitem{KH06}{M. Kokubo and T. Hibi, Weakly Polymatroidal Ideals, {\em Algebra Colloquium} {\bf 13} (2006), 711--720}

\bibitem{Ma20}{ P. Mantero, The structure and free resolutions of the symbolic powers of star configurations of hypersurfaces,
	{\em Trans. Amer. Math. Soc.} {\bf 373} (2020), 8785--8835.}

\bibitem{MN25}{ P. Mantero and V. Nguyen, The structure of symbolic powers of matroids. To appear in 	{\em Trans. Amer. Math. Soc.}.}

\bibitem{Ma73}{ S. B. Maurer, Matroid Basis Graphs. I, 
	{\em J. Combinatorial Theory Ser. B} {\bf 14} (1973), 216--240.}

	\bibitem{MTr09}{ N.C. Minh, N.V. Trung, Cohen–Macaulayness of powers of two-dimensional squarefree monomial ideals, {\em J. Algebra} {\bf 322} (2009) 4219--4227.}	


	\bibitem{MT17}{ N.C. Minh, T. N. Trung,   Regularity of symbolic powers and arboricity of matroids, {\em Forum Math.}  {\bf 31} (2019), 465--477.}
	
	\bibitem{MT11}{ N. C. Minh and N. V. Trung, Cohen-Macaulayness of monomial ideals and symbolic powers of Stanley-Reisner ideals, {\em Adv. Math.} {\bf 226} (2011), 1285--1306.}


    \bibitem{MTT19}{ N. C. Minh. N. Terai, and P. T. Thuy, Level property of ordinary and symbolic powers of {S}tanley-{R}eisner ideals, {\em J. Algebra} {\bf 535} (2019), 350--364.}


\bibitem{Mo11}{  F. Mohammadi, Powers of the vertex cover ideal of a chordal graph. {\em Comm. Algebra} {\bf 39} (2011), 3753--3764.}


\bibitem{MK16}{ S. Moradi and F. Khosh--Ahang, On vertex decomposable simplicial complexes and their Alexander duals, {\em Math. Scand.} {\bf 118} (2016), 43--56.}

	
	\bibitem{MT09}{ S. Murai and N. Terai, {$h$}-vectors of simplicial complexes with {S}erre's
		conditions, {\em Math. Res. Lett.} {\bf 16} (2009), 1015--1028.}	

\bibitem{Na98}{ U. Nagel, Even liaison classes generated by Gorenstein linkage, {\em J. Algebra} {\bf 209} (1998), 543--584.}

\bibitem{NR08}{ U. Nagel and T. R\"omer, Glicci simplicial complexes, {\em J. Pure Appl. Algebra} {\bf 212} (2008), 2250--2258.}

\bibitem{Ox92}{ J.G. Oxley, Matroid theory, {\em Oxford graduate texts in mathematics}, Oxford University Press, 2006.}
 
	\bibitem{Pe11}{ I. Peeva, Graded syzygies. {\em Algebra and Applications} {\bf 14}. Springer-Verlag London, Ltd., London, 2011.} 

\bibitem{PS74}{ C. Peskine and L. Szpiro, Liaison des vari\'et\'es alg\'ebriques, I, {\em Invent. Math.} {\bf 26} (1974), 271--302.}

\bibitem{PU98}{C. Polini and B. Ulrich, Linkage and reduction numbers, {\em Math. Ann.} {\bf 310} (1998), 631--651.}

	\bibitem{Re76}{ G. A. Reisner, Cohen-Macaulay Quotients of Polynomial Rings, {\em Advances in Math.} {\bf 21} (1976), 30--49.}

        \bibitem{RTY11}{ G. Rinaldo, N. Terai and K. Yoshida, On the second powers of Stanley-Reisner ideals, {\em J. Commutative Algebra} {\bf 3} (2011), 405--430.}

\bibitem{Sc82}{ P. Schenzel, Notes on liaison and duality, {\em J. Math. of Kyoto Univ.} {\bf 22} (1982), 485--498.}

\bibitem{Sr89}{  H. Srinivasan, Algebra structures on some canonical resolutions, {\em J. Algebra} {\bf 122} (1989),
	150--187.}

\bibitem{St77}{ R. P. Stanley, Cohen-Macaulay complexes, {\em NATO Adv. Study Inst. Ser. C: Math. Phys. Sci.} {\bf 31}, D. Reidel Publishing Co., Dordrecht-Boston, Mass., 1977, 51--62.} 
 
 
	\bibitem{Ta05}{	Yukihide Takayama, Combinatorial characterizations of generalized
		{C}ohen-{M}acaulay monomial ideals, {\em Bull. Math. Soc. Sci. Math. Roumanie
	(N.S.)} \textbf{48} (2005), no.~3, 327--344. \MR{2165349}}


\bibitem{Te99}{ N. Terai, Alexander duality theorem and Stanley--Reisner rings, Free resolutions of coordinate rings of projective varieties and related topics (Japanese) (Kyoto, 1998), 
S\tovl{u}rikaisekikenky\tovl{u}sho
K\tovl{o}ky\tovl{u}roku, No. 1078 (1999), 174--184.}

\bibitem{Te07}{ N. Terai, Alexander duality in Stanley-Reisner rings, {\em Affine algebraic geometry, Osaka Univ. Press}, Osaka (2007), pp. 449--462.}
	
	\bibitem{TT12}{ N. Terai and N. V. Trung, Cohen-{M}acaulayness of large powers of {S}tanley-{R}eisner ideals, {\em Adv. Math.} {\bf 229} (2012), 711--730.}	
	
	\bibitem{Va11}{ M. Varbaro, Symbolic powers and matroids, {\em Proc. Amer. Math. Soc. } {\bf 139} (2011), 2357--2366.}		
	
	\bibitem{Wa78}{ R. Waldi, Vollst\"andige Durchschnitte in Cohen--Macaulay Ringen, {\em Arch. Math. (Basel)} {\bf  31} (1978/79) 439--44.}	
	
	\bibitem{Zh07}{ W. Zhang, On the highest Lyubeznik number of a local ring, {\em Compos. Math.} {\bf	143} (2007), 82--88.}
\end{thebibliography}

\bigskip
\smallskip
\end{document}